\documentclass[a4paper,10pt]{article}

\usepackage{amssymb,amsfonts}
\usepackage{amsmath,amsfonts,amsthm}
\usepackage[dvipdfmx]{graphicx}
\usepackage[dvipdfmx]{hyperref}
\usepackage{cite}
\usepackage{caption}
\usepackage{subcaption}
\usepackage{mathtools}
\usepackage{booktabs}
\usepackage{cite}
\usepackage[T1]{fontenc}
\usepackage[utf8]{inputenc}
\newcommand{\pd}[2]{\frac{\partial #1}{\partial #2}}
\newcommand{\Dt}[1]{\frac{\textnormal{D} #1}{\textnormal{D}t}}
\newcommand{\di}{\textnormal{div\,}}
\newcommand{\D}[1]{\textnormal{D}(#1)}
\newcommand{\vcko}{\mathbf{v}}
\newcommand{\ucko}{\mathbf{u}}

\newcommand{\wcko}{\mathbf{w}}
\newcommand{\nko}{\mathbf{n}}
\newcommand{\Ccko}{\mathbf{C}}
\newcommand{\Dcko}{\mathbf{D}}
\newcommand{\Icko}{\mathbf{I}}
\newcommand{\Fko}{\mathbf{F}}
\newcommand{\fko}{\mathbf{f}}
\newcommand{\gcko}{\mathbf{g}}
\newcommand{\tr}[1]{\textnormal{tr}\, #1}
\newcommand{\trC}{\textnormal{tr}\, \mathbf{C}}
\newcommand{\n}[2]{\left\lVert{#2}\right\rVert_{#1}}
\newcommand{\trian}{\mathcal{T}_h}
\newcommand{\au}{\ucko_h^n}
\newcommand{\aou}{\ucko_h^{n-1}}
\newcommand{\ap}{p_h^n}

\newcommand{\aC}{\Ccko_h^n}
\newcommand{\aoC}{\Ccko_h^{n-1}}

\newcommand{\atrC}{\tr{\Ccko}_h^n \, }
\newcommand{\aotrC}{\tr{\Ccko}_h^{n-1} \, }

\newcommand{\af}{\fko_h^n}
\newcommand{\aF}{\Fko_h^n}

\newcommand{\SPu}{\hat{\ucko}_h}
\newcommand{\SPp}{\hat{p}_h}
\newcommand{\SPC}{\hat{\Ccko}_h}
\renewcommand{\(}{\left(}
\renewcommand{\)}{\right)}
\newcommand{\defeq}{\vcentcolon=}
\newcommand{\eqdef}{=\vcentcolon}
\newtheorem{defin}{Definition}
\newtheorem{theo}{Theorem}
\newtheorem{prop}{Proposition}
\newtheorem{hypo}{Hypothesis}
\newtheorem{lemma}{Lemma}
\newtheorem{remark}{Remark}
\newtheorem*{example}{Example}
\newcommand{\fz}{\frac}
\newcommand{\prz}[2]{ \frac{\partial{#1}}{\partial{#2}} }
\newcommand{\pz}{\partial}
\newcommand{\lA}{\langle}
\newcommand{\rA}{\rangle}

\newcommand{\ol}{\overline}
\newcommand{\barO}{\bar{\Omega}}

\renewcommand{\Omega}{\varOmega}
\renewcommand{\Gamma}{\varGamma}
\renewcommand{\Psi}{\varPsi}
\renewcommand{\Pi}{\varPi}
\newcommand{\ecko}{\mathbf{e}}
\newcommand{\rcko}{\mathbf{r}}
\newcommand{\Ecko}{\mathbf{E}}
\newcommand{\Rcko}{\mathbf{R}}
\newcommand{\etacko}{{\boldsymbol\eta}}
\newcommand{\Xicko}{{\boldsymbol {\rm\Xi}}}
\newcommand{\DIM}{d}
\allowdisplaybreaks[4]
\DeclareCaptionLabelFormat{tableandpicture}{ \tablename\, \& #1~#2}

\providecommand{\keywords}[1]{\textit{Keywords:} #1}
\providecommand{\msc}[1]{\textit{2010 MSC:} #1}

\title{Numerical analysis of the {Oseen}-type {Peterlin} viscoelastic model by the stabilized {Lagrange}--{Galerkin} method
\\
Part~II: A linear scheme}
\author{
{M\'{a}ria Luk\'{a}\v{c}ov\'{a}-Medvid'ov\'{a}}$^{1}$,
{Hana~Mizerov\'{a}}$^{1}$,\\
{Hirofumi~Notsu}$^{2,3}$
and
{Masahisa~Tabata}$^{4}$
\bigskip\\
\normalsize
$^1$ Institute of Mathematics, University of Mainz, Mainz 55099, Germany \\
\normalsize
$^2$ Faculty of Mathematics and Physics, Kanazawa University, Kanazawa 920-1192, Japan \\
\normalsize
$^3$ Japan Science and Technology Agency~(JST), PRESTO, Saitama 332-0012, Japan\\
\normalsize
$^4$ Department of Mathematics, Waseda University, Tokyo 169-8555, Japan
}
\date{}
\begin{document}
\maketitle
\begin{abstract}
This is the second part of our error analysis of the
stabilized {Lagrange}--{Galerkin} scheme applied to the {Oseen}-type {Peterlin} viscoelastic model.
Our scheme is a combination of the method of characteristics and {Brezzi}--{Pitk\" aranta}'s stabilization method
for the conforming linear elements, which leads to an efficient computation with a small number of degrees of freedom especially in three space dimensions.
In this paper, Part~II, we apply a semi-implicit time discretization which yields the linear scheme.
We concentrate on the diffusive viscoelastic model,
i.e.~in the constitutive equation for time evolution of the conformation tensor a diffusive effect is included.
Under mild stability conditions we obtain error estimates with the optimal convergence order for the velocity, pressure and conformation tensor in two and three space dimensions.
The theoretical convergence orders are confirmed by numerical experiments.
\smallskip\\
\msc{65M12, 76A05, 65M60, 65M25}
\end{abstract}
%
%
%
%
%
%
%
\section{Introduction}\label{sec:intro}
The present paper is a continuation of numerical error analysis of the stabilized {Lagrange}--{Galerkin} method applied to
the {Oseen}-type {Peterlin} viscoelastic model. In our previous paper \cite{LMNT-Peterlin_Oseen_Part_I}, Part~I, we dealt with the fully nonlinear implicit scheme, whereas here, in Part~II, we investigate a linear semi-implicit scheme.
\par
The development of stable and convergent numerical methods for viscoelastic models, such as the {Oldroyd-B} type models, is an active research area.
In particular, the question of stability when elastic effects are dominant (the so-called high Weissenberg number problem) remains an open problem.
We refer the reader to works of {Fattal} and {Kupferman}~\cite{FatKup-2004,FatKup-2005}, where an interesting approach using the log-conformation representation has been introduced. Furthermore, in {Boyaval} et al.~\cite{BoyLelMan-2009} free energy dissipative {Lagrange}--{Galerkin} schemes with or without the log-conformation representation has been studied and in {Lee} and {Xu}~\cite{LeeXu-2006} and {Lee} et al.~\cite{LeeXuZha-2011} finite element schemes using the idea of the generalized {Lie} derivative have been proposed.
Further related numerical schemes and computations can be found, e.g., in~\cite{AboMatWeb-2002,AOP-2003,BonPicLas-2006,BPS-2001,CroKeu-1982,Keu-1986,MarCro-1987,NadSeq-2007,LNS-2015,OP-2002,OCP-2002,WapKeuLeg-2000,W-2013}, see also references therein.
To the best of our knowledge there are no results on error estimates of numerical schemes for the {Oldroyd-B} model, see
{Picasso} and {Rappaz}~\cite{PicRap-2001} and {Bonito} et al.~\cite{BonClePic-2007} for error analysis of simplified models without convective terms.

\par
In~\cite{Pet-1966} Peterlin proposed a mean-field closure model according to which the average of the elastic force over thermal fluctuations is replaced by the value of the force at the mean-squared polymer extension.
This means that a nonlinear spring force law~$F(R) = \gamma(|R|^2) R$  that acts in a dumbbell is replaced by the function~$F(R)=\gamma (\tr \Ccko) R$.
Here, $\gamma$ is the spring constant, $\Ccko$ is the so-called conformation tensor and $R$ is the vector connecting the beads of a dumbbell.
Based on this approach {Renardy} has recently derived a new class of general macroscopic constitutive models, that is motivated by Peterlin dumbbell theories with a nonlinear spring law for an infinitely extensible
spring, see Renardy~\cite{RenWan-2015, Ren-2010} and recent papers by {Luk\'{a}\v{c}ov\'{a}-Medvi\v{d}ov\'{a}} et al.~\cite{LukMizNec-2015, LMNR-2016}, where the global existence of weak solutions has been obtained.

In this paper, Part~II, as well as in our previous paper, Part~I, we consider the so-called Oseen-type Peterlin viscoelastic model
that is a system of the flow equations and an equation for the conformation tensor, cf.~\cite{Ren-2008, RenWan-2015, Ren-2010}.
We concentrate on the diffusive viscoelastic model, which means that in the constitutive equations for the conformation tensor a diffusive effect is included.

Let us point out that in standard derivations of bead-spring models the diffusive term in the equation for the elastic stress tensor is routinely omitted.
In \cite{DL-2009} a careful justification of the presence of the diffusive term in the Fokker-Planck equations through the asymptotic analysis is presented.
The diffusion coefficient~$\varepsilon$ is proportional to $(\ell/L)^2/We$, where $L$ and $\ell$ are characteristic macroscopic and microscopic length scales, respectively, and $We$ is the so-called {Weissenberg} number.
It is a reference number characterizing viscoelastic property of the material.
Estimates for $(\ell/L)^2$ presented in \cite{BAB-1991} show that $(\ell/L)^2$ is in the range of about $10^{-9}$ and $10^{-7}$.
As emphasized in \cite{BaSu-2012} the model reduction by neglecting this small diffusive effect is mathematically counterproductive leading to a degenerate parabolic-hyperbolic system~\eqref{model} with $\varepsilon = 0$.
On the other hand, when the diffusive term is taken into account, the resulting system~\eqref{model} remains parabolic. We would like to point out that in the analysis presented below we only require $\varepsilon > 0$ and there is no assumption on the size of $\varepsilon$.
For the details of the derivation of the diffusive {Peterlin} model we refer to~\cite{LukMizNec-2015,Miz-2015,RenWan-2015,Ren-2010}.
Let us mention that, even when the velocity field is given, the equation for the conformation tensor in the {Peterlin} model is still nonlinear, while the {Oldroyd-B} model is linear with respect to the extra stress tensor.
Hence, we can say that the nonlinearity of the {Peterlin} model is stronger than that of the {Oldroyd-B} model.
As a starting point of the numerical analysis of the {Peterlin} model, we consider the {Oseen}-type model, where the velocity of the material derivative is replaced by a known one, in order to concentrate on the treatment of nonlinear terms arising from the elastic stress.
\par
In the present paper a stabilized {Lagrange}--{Galerkin} method for the {Peterlin} viscoelastic model is studied.
It consists of the method of characteristics and {Brezzi}--{Pitk\"{a}ranta}'s stabilization method~\cite{BrePit-1984} for the conforming linear elements.  The method of characteristics derives the robustness in convection-dominated flow problems,
and the stabilization method reduces the number of degrees of freedom in computation especially in three space dimensions.
In our recent works by {Notsu} and {Tabata}~\cite{NT-2016-M2AN,NT-2015-JSC,NT-NCP} the stabilized {Lagrange}--{Galerkin} method has been applied successfully for the {Oseen}, {Navier}--{Stokes} and natural convection problems and optimal error estimates have been proved.
We extend the numerical analysis of the stabilized {Lagrange}--{Galerkin} method to the {Oseen}-type {Peterlin} model.
As already mentioned above, the aim of the present paper paper is to give a rigorous error analysis of the linear stabilized {Lagrange}--{Galerkin} scheme for the diffusive {Peterlin} model in both two and three space dimensions.
We show that under mild stability conditions the obtained error estimates have the optimal convergence rate.
\par
As mentioned in {Boyaval} et al.~\cite{BoyLelMan-2009}, the positive definiteness of the conformation tensor is important in the analysis of numerical schemes for the viscoelastic models, where this property has been shown for the exact strong solution in~\cite{Miz-2015}.
We remark that our error estimates have been obtained successfully without studying positive definiteness of the conformation tensor.
Let us additionally note that this paper includes the error estimate for the pressure in the standard $L^2$ norm (Theorem 2), which has, as far as we know, never been shown for time-dependent viscoelastic flow problems, e.g., the Oldroyd-B model.
\par
This paper is organized as follows.
In Section~\ref{sec:model} the mathematical model for the {Peterlin} viscoelastic fluid is described.
In Section~\ref{sec:scheme} a linear stabilized {Lagrange}--{Galerkin} scheme is presented.
The main results on the convergence with optimal error estimates are stated in Section~\ref{sec:main_results}, and proved in Section~\ref{sec:proofs}.
In Section~\ref{sec:numerics} some numerical experiments confirming the theoretical convergence orders are provided.
%
%
%
%
%
%
\section{The {Oseen}-type {Peterlin} viscoelastic model}\label{sec:model}
The function spaces and the notation to be used throughout the paper are as follows.
Let $\Omega$ be a bounded domain in $\mathbb{R}^\DIM$ for $\DIM =2$ or $3$, $\Gamma\defeq\pz\Omega$ the boundary of $\Omega$, and $T$ a positive constant.
For $m \in \mathbb{N}\cup \{0\}$ and $p\in [1,\infty]$ we use the {\rm Sobolev} spaces $W^{m,p}(\Omega)$, $W^{1,\infty}_0(\Omega)$, $H^m(\Omega) \, (=W^{m,2}(\Omega))$, $H^1_0(\Omega)$
and
$L^2_0(\Omega)\defeq\{q\in L^2(\Omega); \int_\Omega q~dx =0\}$.
Furthermore, we employ function spaces $H^m_{sym}(\Omega) \defeq \{\Dcko\in H^m (\Omega)^{\DIM \times \DIM};~\Dcko=\Dcko^T\}$ and $C^m_{sym}(\barO) \defeq C^m(\barO)^{\DIM \times \DIM}\cap H^m_{sym}(\Omega)$, where the superscript~$T$ stands for the transposition.
For any normed space~$S$ with norm~$\|\cdot\|_S$, we define function spaces $H^m(0,T; S)$ and $C([0,T]; S)$ consisting of $S$-valued functions in $H^m(0,T)$ and $C([0,T])$, respectively.
We use the same notation $(\cdot, \cdot)$ to represent the $L^2(\Omega)$ inner product for scalar-, vector- and matrix-valued functions.
The dual pairing between $S$ and the dual space $S^\prime$ is denoted by $\lA\cdot, \cdot\rA$.
The norms on $W^{m,p}(\Omega)$ and $H^m(\Omega)$ and their seminorms are simply denoted by $\|\cdot\|_{m,p}$ and $\|\cdot\|_m \, (= \|\cdot\|_{m,2})$ and by $|\cdot|_{m,p}$ and $|\cdot|_m \, (= |\cdot|_{m,2})$, respectively.
The notations~$\|\cdot\|_{m,p}$, $|\cdot|_{m,p}$, $\|\cdot\|_m$ and $|\cdot|_m$ are employed not only for scalar-valued functions but also for vector- and matrix-valued ones.
We also denote the norm on $H^{-1}(\Omega)^2$ by $\|\cdot\|_{-1}$.
For $t_0$ and $t_1\in\mathbb{R}$ we introduce the function space,
\begin{align*}
Z^m(t_0, t_1) & \defeq \bigl\{ \psi \in H^j(t_0, t_1; H^{m-j}(\Omega));~j=0,\ldots,m,\ \|\psi\|_{Z^m(t_0, t_1)} < \infty \bigr\}
\end{align*}
with the norm
\begin{align*}
\|\psi\|_{Z^m(t_0, t_1)} & \defeq \biggl\{ \sum_{j=0}^m \|\psi\|_{H^j(t_0,t_1; H^{m-j}(\Omega))}^2 \biggr\}^{1/2},
\end{align*}
and set $Z^m \defeq Z^m(0, T)$.
We often omit $[0,T]$, $\Omega$, and the superscripts~$\DIM$ and~$\DIM \times \DIM$ for the vector and the matrix if there is no confusion, e.g., we shall write $C(L^\infty)$ in place of $C([0,T]; L^\infty(\Omega)^{\DIM \times \DIM})$.
For square matrices $\mathbf{A}$ and $\mathbf{B} \in \mathbb{R}^{\DIM \times \DIM}$ we use the notation~$\mathbf{A}:\mathbf{B}
\defeq \tr ( \mathbf{A} \mathbf{B}^T )
= \sum_{i,j} A_{ij} B_{ij}$.
\par
We consider the system of equations describing the unsteady motion of an incompressible viscoelastic fluid,
\begin{subequations}\label{model}
\begin{align}
\Dt{\ucko} - \di \(2\nu\D{\ucko}\) + \nabla p & = \di [(\trC) \Ccko] + \fko & & \mbox{in}~\Omega \times (0,T),  \label{model_ucko}\\
\di \ucko &= 0 & & \mbox{in}~\Omega \times (0,T), \\
\Dt{\Ccko} - \varepsilon\Delta\Ccko =  (\nabla\ucko)\Ccko & + \Ccko(\nabla\ucko)^T - \(\trC\)^2 \Ccko   + (\trC) \Icko + \Fko & & \mbox{in}~\Omega \times (0,T),
\label{model_Ccko}\\
\ucko &= \mathbf{0}, \quad  \pd{\Ccko}{\nko} = \mathbf{0}, & & \mbox{on}~\Gamma \times (0,T),
\label{model_bc}\\
\ucko &=\ucko^0,\quad \Ccko =   \Ccko^0,
 & & \mbox{in}~\Omega,\ \mbox{at}\ t=0,
\label{model_ic}
\end{align}
\end{subequations}
where $\(\ucko, p, \Ccko\): \Omega\times (0,T) \rightarrow \mathbb{R}^\DIM \times \mathbb{R}\times \mathbb{R}^{\DIM \times \DIM}_{sym}$ are the unknown velocity, pressure and conformation tensor, $\nu \in (0, 1]$ is a fluid viscosity, $\varepsilon \in (0, 1]$ is an elastic stress viscosity,
$(\fko, \Fko): \Omega \times (0,T) \rightarrow \mathbb{R}^\DIM \times \mathbb{R}^{\DIM \times \DIM}_{sym}$ is a pair of given external forces,
$\D{\ucko} \defeq (1/2) [\nabla \ucko + (\nabla \ucko)^T]$ is the symmetric part of the velocity gradient,
$\Icko$~is the identity matrix,
$\nko:\Gamma\to\mathbb{R}^\DIM$~is the outward unit normal,
$(\ucko^0, \Ccko^0 ) : \Omega\to \mathbb{R}^\DIM \times \mathbb{R}^{\DIM \times \DIM}_{sym}$ is a pair of given initial functions,
and $\textnormal{D}/\textnormal{D}t$ is the material derivative defined by
\begin{align*}
\Dt{\ } \defeq \pd{\ }{t} + \wcko\cdot\nabla,
\end{align*}
where $\wcko:\Omega\times (0,T) \rightarrow \mathbb{R}^\DIM$ is a given velocity.
\begin{remark}
The model~\eqref{model} is  the {\rm Oseen} approximation to the fully nonlinear problem, where the material derivative terms,
\[
\frac{\partial \ucko}{\partial t}+(\ucko\cdot\nabla)\ucko,
\quad
\frac{\partial \Ccko}{\partial t}+(\ucko\cdot\nabla)\Ccko
\]
exist in place of $\Dt{\ucko}$ and $\Dt{\Ccko}$ in equations~\eqref{model_ucko} and~\eqref{model_Ccko}.
The existence of weak solutions and the uniqueness of regular solutions to the fully nonlinear model have been proved in {\rm Luk\'{a}\v{c}ov\'{a}-Medvid'ov\'{a}} et al.~\cite[Theorems~1 and~3]{LukMizNec-2015}.
The corresponding results are obtained under regularity condition on $\wcko$ to the model~\eqref{model}, which is simpler than the fully nonlinear model.
Numerical analysis of the fully nonlinear problem is a future work.
\end{remark}
\par
We set an assumption for the given velocity~$\wcko$.
\begin{hypo}\label{hyp:w}
The function $\wcko$ satisfies $\wcko \in C([0,T];W^{1,\infty}_0(\Omega)^\DIM).$
\end{hypo}
\par
Let $V \defeq H_0^1(\Omega)^\DIM$, $Q \defeq L_0^2(\Omega)$ and $W \defeq H^1_{sym}(\Omega)$.
We define the bilinear forms $a_u$ on $V \times V,$  $b$ on $V \times Q,$ $\mathcal{A}$ on $(V \times Q)\times(V \times Q)$ and $a_c$ on $W \times W$ by
\begin{align*}
a_u\(\ucko,\vcko\) & \defeq  2 \bigl( \D{\ucko}, \D{\vcko} \bigr), & b(\ucko,q) &\defeq - (\di \ucko, q),
& \mathcal{A}\bigl( (\ucko,p), (\vcko,q) \bigr) & \defeq \nu a_u\( \ucko,\vcko \) + b(\ucko,q) + b(\vcko,p), \\
a_c\(\Ccko,\Dcko\) & \defeq (\nabla\Ccko, \nabla\Dcko),
\end{align*}
respectively.
We present the weak formulation of the problem~\eqref{model};
find $(\ucko,p,\Ccko): (0,T) \rightarrow V \times Q \times W$ such that for $t \in (0,T)$
\begin{subequations}\label{weak_formulation}
\begin{align}
\biggl( \Dt{\ucko}(t),\vcko \biggr) &+ \mathcal{A}\bigl( (\ucko,p)(t), (\vcko,q) \bigr) = -  \(\trC(t)\,\Ccko(t),\nabla\vcko\) + \(\fko(t),\vcko\),
\label{weak_formulation_ucko}\\
\biggl( \Dt{\Ccko}(t), \Dcko \biggr) &+ \varepsilon a_c \bigl( \Ccko(t),\Dcko \bigr) = 2\bigl( (\nabla\ucko(t)) \Ccko(t), \Dcko \bigr) - \bigl( (\trC(t))^2 \Ccko(t),\Dcko \bigr) + \(\trC(t)\Icko,\Dcko\) + \(\Fko(t),\Dcko\),
\label{weak_formulation_Ccko} \\
&
\qquad\qquad\qquad\qquad\qquad\qquad\qquad\qquad\qquad\qquad\qquad\qquad\qquad\qquad
\forall (\vcko,q,\Dcko) \in V\times Q \times W, \notag
\end{align}
\end{subequations}
with $( \ucko(0),  \Ccko(0) ) =( \ucko^0 , \Ccko^0)$.
%
%
%
%
%
%
%
%
%
\section{A linear stabilized {Lagrange}--{Galerkin} scheme}\label{sec:scheme}
The aim of this section is to present a linear stabilized {Lagrange}--{Galerkin} scheme for the model~\eqref{model}.
\par
Let $\Delta t$ be a time increment, $N_T \defeq \lfloor T/\Delta t \rfloor$ the total number of time steps and $t^n \defeq n \Delta t$ for $n=0,\ldots,N_T$.
Let $\gcko$ be a function defined in $\Omega\times (0,T)$ and $\gcko^n \defeq \gcko(\cdot,t^n)$.
For the approximation of the material derivative we employ the first-order characteristics method,
\begin{align}\label{approx_matder}
\Dt{\gcko}(x,t^n) = \frac{\gcko^n(x) - \( \gcko^{n-1} \circ X_1^n\) (x)}{\Delta t} + O(\Delta t),
\end{align}
where $X_1^n:\Omega \to \mathbb{R}^\DIM$ is a mapping defined by
\[
X_1^n(x) \defeq x-\wcko^n(x)\Delta t,
\]
and the symbol~$\circ$ means the composition of functions,
\[
(\gcko^{n-1}\circ X_1^n)(x) \defeq \gcko^{n-1} ( X_1^n(x) ).
\]
For the details on deriving the approximation~\eqref{approx_matder} of $\textnormal{D} \gcko/ \textnormal{Dt},$ see, e.g.,~\cite{NT-2015-JSC}.
The point $X_1^n(x)$ is called the upwind point of~$x$ with respect to~$\wcko^n$.
The next proposition, which is a direct consequence of \cite{RuiTab-2002} and \cite{TabUch-2015-NS},
presents sufficient conditions to ensure that all upwind points defined by $X_1^n$ are in $\Omega$ and that its Jacobian~$J^n \defeq \det ( \pz X_1^n / \pz x )$ is around $1$.
\begin{prop}\label{prop:RT_TU}
Suppose Hypothesis~\ref{hyp:w} holds.
Then, we have the following for $n \in \{0,\ldots,N_T\}$.
\smallskip\\
(i)~Under the condition
\begin{align}
\Delta t |\wcko|_{C(W^{1,\infty})} < 1,
\label{cond:dt_w_bijective}
\end{align}
$X_1^n: \Omega \to \Omega$ is bijective.
\smallskip\\
(ii)~Furthermore, under the condition
\begin{align}
\Delta t |\wcko|_{C(W^{1,\infty})} \le 1/4,
\label{cond:dt_w_Jacobian}
\end{align}
the estimate $1/2 \le J^n \le 3/2$ holds.
\end{prop}
\par
For the sake of simplicity we suppose that $\Omega$ is a polygonal domain.
Let $\trian=\{K\}$ be a triangulation of $\bar{\Omega} \ (= \bigcup_{K\in\mathcal{T}_h} K )$, $h_K$
the
diameter of $K \in \trian$ and $h \defeq \max_{K\in\trian}h_K$ the maximum element size.
We consider a regular family of subdivisions $\{\trian\}_{h\downarrow 0}$ satisfying the inverse assumption~\cite{Cia-1978}, i.e., there exists a positive constant $\alpha_0$ independent of $h$ such that
\begin{align*}
\frac{h}{h_K} \leq \alpha_0, \quad \forall K \in \trian, \ \forall h.
\end{align*}
We define the discrete function spaces $X_h$, $M_h$, $W_h$, $V_h$ and $Q_h$ by
\begin{align*}
X_h &\defeq  \left\{ \vcko_h \in C(\bar{\Omega})^\DIM ; \ \vcko_{h|K} \in P_1(K)^\DIM, \forall K \in \trian \right\}, 
&
M_h & \defeq \left\{ q_h \in C(\bar{\Omega}) ; \ q_{h|K} \in P_1(K), \forall K \in \trian \right\}, 
\\
W_h &\defeq  \left\{ \Dcko_h \in C_{sym}(\bar{\Omega}); \ \Dcko_{h|K} \in P_1(K)^{\DIM \times \DIM}, \forall K \in \trian \right\},
&
V_h &\defeq X_h \cap V, \qquad Q_h \defeq M_h \cap Q,
\end{align*}
respectively, where $P_1(K)$ is the polynomial space of linear functions on~$K\in\trian$.
\par
Let $\delta_0$ be a small positive constant fixed arbitrarily and $(\cdot, \cdot)_K $ the $L^2(K)^\DIM$ inner product.
We define the bilinear forms $\mathcal{A}_h$ on $(V \times H^1(\Omega)) \times (V \times H^1(\Omega))$ and $\mathcal{S}_h$ on $H^1(\Omega) \times H^1(\Omega)$ by
\begin{align*}
\mathcal{A}_h\( (\ucko,p),(\vcko,q) \) &\defeq \nu a_u\( \ucko,\vcko \) + b(\ucko,q) + b(\vcko,p) - \mathcal{S}_h(p,q), &
\mathcal{S}_h(p,q) & \defeq \delta_0\sum_{K\in\trian}h_K^2(\nabla p,\nabla q)_K.
\end{align*}
\par
Let $(\fko_h, \Fko_h) \defeq (\{ \fko_h^n\}_{n=1}^{N_T},$ $\{ \Fko_h^n\}_{n=1}^{N_T} ) \subset L^2(\Omega)^\DIM \times L^2(\Omega)^{\DIM \times \DIM}$  and~$(\ucko_h^0,\Ccko_h^0) \in V_h \times W_h$ be given.
A linear stabilized {\rm {Lagrange}--{Galerkin}} scheme for~\eqref{model} is to find $(\ucko_h, p_h, \Ccko_h) \defeq \{(\au,\ap,\aC)\}_{n=1}^{N_T}$ $\subset V_h \times Q_h \times W_h$ such that, for $n=1,\ldots ,N_T$,
\begin{subequations}\label{lin_scheme}
\begin{align}
\biggl( \frac{\au-\aou\circ X_1^n}{\Delta t}, \vcko_h \biggr) + \mathcal{A}_h \bigl( ( \au,\ap ), (\vcko_h,q_h) \bigr)
& = - \bigl( (\atrC)\aoC,\nabla \vcko_h \bigr) + (\af,\vcko_h),
\label{lin_scheme_velocity_pressure}\\
\biggl( \frac{\aC-\aoC \circ X_1^n}{\Delta t}, \Dcko_h \biggr) + \varepsilon a_{c}(\aC,\Dcko_h)
& = 2\bigl( (\nabla\au)\aoC, \Dcko_h \bigr) - \bigl( (\aotrC)^2 \aC, \Dcko_h \bigr) 
\notag\\
& \qquad\qquad\qquad + \bigl( (\aotrC)\Icko,\Dcko_h \bigr) + (\aF,\Dcko_h),
\label{lin_scheme_tensor} \\
& \qquad\qquad\qquad\qquad
\forall (\vcko_h,q_h,\Dcko_h) \in V_h\times Q_h \times W_h. \notag
\end{align}
\end{subequations}
%
%
%
%
%
%
%
%
\section{The main result}\label{sec:main_results}
In this section we state the main result on error estimates with the optimal convergence order of scheme~\eqref{lin_scheme},
which is proved in the next section.
\par
We use $c$, $c_w$, $c_s$, $c_{w,s}$, $c_\nu$, $c_\varepsilon$ and $c_{\nu,\varepsilon}$ to represent generic positive constants independent of the discretization parameters $h$ and $\Delta t$, the subscripts imply the dependency of the constants, and the subscripts {``~$w$~''} and {``~$s$~''} in $c_w$, $c_s$ and $c_{w,s}$ mean the given velocity~$\wcko$ and the solution $(\ucko, p, \Ccko)$ of~\eqref{weak_formulation}, respectively.
For instance, the constant~$c_{w,s}$ is dependent on $\wcko$ and $(\ucko, p, \Ccko)$ and independent of $\nu$ and $\varepsilon$, and the constant~$c$ has no dependency on $\wcko$, $(\ucko, p, \Ccko)$, $\nu$ nor $\varepsilon$.
The symbol ``$\prime$ (prime)'' is sometimes used in order to distinguish two constants, e.g.,~$c_s$ and~$c_s^\prime$, from each other.
\par
We use the following notation for the norms and seminorms, $\n{V}{\cdot} = \n{V_h}{\cdot} \defeq \n{1}{\cdot}$, $\n{Q}{\cdot} = \n{Q_h}{\cdot} \defeq \n{0}{\cdot}$,
\begin{align*}
& \n{Z^2(t_0,t_1)}{(\ucko,\Ccko)} \defeq \Bigl\{ \n{Z^2(t_0,t_1)}{\ucko}^2 + \n{Z^2(t_0,t_1)}{\Ccko}^2 \Bigr\}^{1/2}, \\
& \n{\ell^{\infty}(X)}{\ucko} \defeq \max_{n=0,\ldots,N_T} \n{X}{\ucko^n},
\qquad 
\|\ucko\|_{\ell^2_m(X)} \defeq \biggl\{ \Delta t\sum_{n=1}^{m} \|\ucko^n\|_X^2 \biggr\}^{1/2},
\qquad
\n{\ell^2(X)}{\ucko} \defeq \|\ucko\|_{\ell^2_{N_T}(X)}, \\
&
|p|_h \defeq \biggl\{ \sum_{K \in \trian} h_K^2 ( \nabla p, \nabla p )_K \biggr\}^{1/2},
\qquad
|p|_{\ell^2_m(|.|_h)} \defeq \biggl\{ \Delta t \sum_{n=1}^m |p^n|_h^2 \biggr\}^{1/2},
\qquad
|p|_{\ell^2(|.|_h)} \defeq |p|_{\ell^2_{N_T}(|.|_h)},
\end{align*}
for $m\in\{1,\cdots,N_T\}$ and $X=L^\infty(\Omega)$, $L^2(\Omega)$ and~$H^1(\Omega)$.
$\ol{D}_{\Delta t}$ is the backward difference operator defined by $\ol{D}_{\Delta t} \ucko^n \defeq (\ucko^n - \ucko^{n-1})/\Delta t$.
\par
The existence and uniqueness of the solution of scheme~\eqref{lin_scheme} are ensured by the following proposition, which is also proved in the next section.
\begin{prop}[existence and uniqueness]\label{prop:existence_uniqueness}
Suppose Hypothesis~\ref{hyp:w} holds.
Then, for any~$h$ and~$\Delta t$ satisfying~\eqref{cond:dt_w_bijective} there exists a unique solution $(\ucko_h, p_h, \Ccko_h) \subset V_h \times Q_h \times W_h$ of scheme~\eqref{lin_scheme}.
\end{prop}
%
%
\par
We state the main results after preparing a projection and a hypothesis.
\begin{defin}[{Stokes}--{Poisson} projection]
For $(\ucko, p, \Ccko) \in V \times Q \times W $ we define the {\rm {Stokes}--{Poisson}} projection $(\SPu, \SPp, \SPC) \in V_h \times Q_h \times W_h$ of $(\ucko, p, \Ccko)$ by
\begin{align}
\mathcal{A}_h\((\SPu,\SPp),(\vcko_h, q_h)\) + a_c (\SPC, \Dcko_h) + (\SPC, \Dcko_h)
= \mathcal{A}\((\ucko,p),(\vcko_h, q_h)\) + a_c (\Ccko, \Dcko_h) + (\Ccko, \Dcko_h),
\notag \\
\forall (\vcko_h, q_h, \Dcko_h) \in V_h \times Q_h \times W_h.
\label{StokesPoisson_projection}
\end{align}
\end{defin}
\noindent
The {\rm {Stokes}--{Poisson}} projection derives an operator
$\Pi_h^{\rm SP}: V \times Q \times W \to V_h\times Q_h \times W_h$ defined by $\Pi_h^{\rm SP} (\ucko, p, \Ccko) := (\hat{\ucko}_h, \hat{p}_h, \hat{\Ccko}_h)$.
We denote the $i$-th component of $\Pi_h^{\rm SP} (\ucko, p, \Ccko)$ by $[\Pi_h^{\rm SP} (\ucko, p, \Ccko)]_i$ for $i=1,2,3$ and the pair of the first and third components~$(\hat{\ucko}_h, \hat{\Ccko}_h)=([\Pi_h^{\rm SP} (\ucko, p, \Ccko)]_1, [\Pi_h^{\rm SP} (\ucko, p, \Ccko)]_3)$ by $[\Pi_h^{\rm SP} (\ucko, p, \Ccko)]_{1,3}$ simply.
\begin{remark}
The identity~\eqref{StokesPoisson_projection} can be decoupled into the {\rm Stokes} projection and the {\rm Poisson} projection.
For the simplicity of the notation we use~\eqref{StokesPoisson_projection} in the sequel.
Since the Neumann boundary condition~\eqref{model_bc} is imposed on~$\Ccko$, we use the {\rm Poisson} projection corresponding to the operator~$-\Delta + I$ for the unique solvability.
\end{remark}
\begin{hypo}\label{hyp:regularity}
The solution $\(\ucko, p, \Ccko \) $ of~\eqref{weak_formulation} satisfies
$\ucko \in Z^2(0,T)^\DIM \cap H^1(0,T; V \cap H^2(\Omega)^\DIM)
\cap C([0,T]; W^{1,\infty}(\Omega)^\DIM)$,
$p \in H^1(0,T; Q \cap H^1(\Omega))$ and 
$\Ccko \in Z^2(0,T)^{\DIM \times \DIM} \cap H^1(0,T;W\cap H^2(\Omega)^{\DIM \times \DIM})$.
\end{hypo}
\begin{remark}
Let us note that we assume a higher regularity of the exact solution than that of the weak solution.
Such regularity is usually assumed in discussing the convergence rate of numerical solutions of partial differential equations.
We remark that our recent theoretical result~\cite{LukMizNec-2015} shows that both velocity and conformation tensor belong to $L^{\infty}(H^2)$ for the fully nonlinear {\rm Peterlin} model with $\varepsilon>0$.
The result holds also for the {\rm Oseen}-type {\rm Peterlin} model with $\varepsilon>0$.
\end{remark}
We now impose the conditions
\begin{align}
(\ucko_h^0,\Ccko_h^0) = [\Pi_h^{\rm SP} (\ucko^0, 0, \Ccko^0)]_{1,3}, \quad (\fko_h, \Fko_h)=(\fko,\Fko).
\label{cond:if}
\end{align}
\begin{remark}
For the choice of~$(\ucko_h^0,\Ccko_h^0)$ we employ the {\rm Stokes--Poisson} projection of $(\ucko^0, 0, \Ccko^0)$ by~\eqref{StokesPoisson_projection} in~\eqref{cond:if}, since the initial condition for the pressure is not given in~\eqref{model}.
This choice does not lose any convergence order in our results below.
\end{remark}
%
%
%
\begin{theo}[error estimates~I]\label{thm:error_estimates}
Suppose Hypotheses~\ref{hyp:w} and~\ref{hyp:regularity} hold.
Then, there exist positive constants $h_0$, $c_0$ and $c_\dagger$ such that, for any pair~$(h, \Delta t)$ satisfying
\begin{align}
h\in (0,h_0],\quad
\Delta t\le
\left\{
\begin{aligned}
& c_0 (1+|\log h|)^{-1/2} && (d=2),\\
& c_0 h^{1/2} && (d=3),
\end{aligned}
\right.
\label{condition:h_dt}
\end{align}
the solution $(\ucko_h, p_h, \Ccko_h)$ of scheme~\eqref{lin_scheme} with~\eqref{cond:if} is estimated as follows.
\begin{align}
\|\Ccko_h\|_{\ell^\infty(L^\infty)} & \le \|\Ccko\|_{C(L^\infty)}+1,
\label{ieq:stability} \\
\|\ucko_h-\ucko\|_{\ell^\infty(L^2)},\ 
\|\ucko_h-\ucko\|_{\ell^2(H^1)},\ 
|p_h-p|_{\ell^2(|\cdot|_h)}, &\ 
\|\Ccko_h-\Ccko\|_{\ell^\infty(H^1)},\ 
\Bigl\|\ol{D}_{\Delta t}\Ccko_h-\prz{\Ccko}{t}\Bigr\|_{\ell^2(L^2)} 
\le c_\dagger (\Delta t + h).
\label{ieq:main_results}
\end{align}
\end{theo}
\begin{theo}[error estimates~II]\label{thm:error_estimates_pressure}
Suppose Hypotheses~\ref{hyp:w} and~\ref{hyp:regularity} hold.
Let $h_0$ and $c_0$ be the constants stated in Theorem~\ref{thm:error_estimates}.
Then, there exists a positive constant $c_\ddagger$ such that, for any pair~$(h, \Delta t)$ with~\eqref{condition:h_dt}
the solution $(\ucko_h, p_h, \Ccko_h)$ of scheme~\eqref{lin_scheme} with~\eqref{cond:if} satisfies the estimates,
\begin{align}
\biggl\|\ol{D}_{\Delta t}\ucko_h-\prz{\ucko}{t}\biggr\|_{\ell^2(L^2)},\ \ \|p_h-p\|_{\ell^2(L^2)} \le c_\ddagger (\Delta t + h).
\label{ieq:main_results_pressure}
\end{align}
\end{theo}
\begin{remark}
(i)~The condition~\eqref{condition:h_dt} is mild in comparison with, e.g., the CFL condition of the form~$\|\wcko\|_{C(L^\infty)}\Delta t \le c h$.
We can take $\Delta t = c h^{\alpha}$ for any $\alpha > 0~(d=2)$ or $\alpha \ge 1/2~(d=3)$.
\smallskip\\
(ii)~The condition~\eqref{condition:h_dt} is needed to deal with the nonlinearity of the model or, more precisely, to get the boundedness of $\|\Ccko_h^n\|_{0,\infty}$ by using the inverse inequality~\eqref{ieq:inverse}, cf. the estimate~\eqref{ieq:Ch_0_infty} with~\eqref{def:c0_h0_Linf}.
In fact, the stabilized {\rm Lagrange}--{\rm Galerkin} scheme for the {\rm Oseen} equations is stable under only~\eqref{cond:dt_w_Jacobian}, cf.~\cite{NT-2015-JSC}.
\end{remark}
%
%
%
%
%
%
\section{Proofs}\label{sec:proofs}
In what follows we prove Proposition~\ref{prop:existence_uniqueness} and Theorems~\ref{thm:error_estimates} and~\ref{thm:error_estimates_pressure}.
%
\subsection{Preliminaries}
Let us list lemmas employed directly in the proofs below.
In the lemmas, $\alpha_i$, $i=1,\ldots,4$, are numerical constants independent of~$h$, $\Delta t$, $\nu$ and~$\varepsilon$.
\begin{lemma}[ \cite{Nec-1967} ]\label{lem:Korn}
Let $\Omega$ be a bounded domain with a {\rm Lipschitz}-continuous boundary.
Then, the following inequalities hold.
\begin{align*}
\|\D{\vcko}\|_0 \le \|\vcko\|_1 \le \alpha_1 \|\D{\vcko}\|_0,\qquad \forall \vcko \in H^1_0(\Omega)^\DIM.
\end{align*}
\end{lemma}
Let $\Pi_h: C(\barO)\to M_h$ be the {Lagrange} interpolation operator.
The operators defined on~$C(\barO)^\DIM$ and~$C(\barO)^{\DIM \times \DIM}$ are also denoted by the same symbol~$\Pi_h$.
We introduce the function
\begin{align}
D(h) \defeq
\left\{
\begin{aligned}
& (1+|\log h|)^{1/2} && (d=2),\\
& h^{-1/2} && (d=3),
\end{aligned}
\right.
\label{Dofh}
\end{align}
which is used in the sequel.
\begin{lemma}[ \cite{BreSco-2008,Cia-1978} ]\label{linear_operators}
The following inequalities hold.
\begin{align}
 \n{0,\infty}{\Pi_h\Dcko} &\leq \n{0,\infty}{\Dcko}, & \forall \Dcko &\in C(\barO)^{\DIM \times \DIM}, \notag\\
 \n{1}{\Pi_h \Dcko - \Dcko} &\leq \alpha_{20} h \n{2}{\Dcko}, & \forall \Dcko & \in H^2(\Omega)^{\DIM \times \DIM}, \notag\\
 \n{0,\infty}{\Dcko_h} &\leq \alpha_{21} D(h) \n{1}{\Dcko_h}, & \forall \Dcko_h & \in W_h.
 \label{ieq:inverse}
\end{align}
\end{lemma}
\par
The next lemma is obtained by combining the error estimates for the {Stokes} and the {Poisson} problems, see, e.g.,~\cite{BreDou-1988,Cia-1978,FraSte-1991} for the proof.
\begin{lemma} \label{lem:estimates_SP_projection}
(i)~The following inequality holds.
\begin{align*}
\inf_{(u_h,p_h)\in V_h\times Q_h}\sup_{(v_h,q_h)\in V_h\times Q_h} \fz{\mathcal{A}_h((u_h,p_h), (v_h,q_h))}{\|(u_h,p_h)\|_{V\times Q} \|(v_h,q_h)\|_{V\times Q}}\ge \nu\alpha_{30}.
\end{align*}
(ii)~Assume $(\ucko, p, \Ccko) \in (V \cap H^2(\Omega)^\DIM) \times (Q \cap H^1(\Omega)) \times (W \cap H^2(\Omega)^{\DIM \times \DIM})$.
Let $(\hat{\ucko}_h, \hat{p}_h, \hat{\Ccko}_h) \in V_h \times Q_h \times W_h$ be the {\rm {Stokes}--{Poisson}} projection of $(\ucko, p, \Ccko)$ defined by~\eqref{StokesPoisson_projection}.
Then, the following inequalities hold.
\begin{align*}
\n{1}{\SPu-\ucko}, \ \n{0}{\SPp-p}, \ |\SPp-p|_h & \le \fz{\alpha_{31}}{\nu} h \n{H^2 \times H^1 }{(\ucko,p)},
&
\|\SPC-\Ccko\|_1 & \le \alpha_{32} h \|\Ccko\|_2.
\end{align*}
\end{lemma}
\begin{remark}\label{rmk:alpha3}
Let us note that the first part of error estimates in~(ii) is based on the generalized inf-sup condition in~(i) that is satisfied by the bilinear form $\mathcal{A}_h$ defined above and the pair of the discrete function spaces~$V_h$ and~$Q_h$, where the $\nu$-dependency is obtained by a simple modification of the analysis in, e.g., \cite{FraSte-1991} after taking into account the diffusion constant.
\end{remark}
\begin{remark}
As pointed out in~\cite{BPS-2001}, there are basically three possible approaches to obtain stable and convergent numerical methods for viscoelastic fluid flow problems. Firstly, the usual {\rm Galerkin} methods using finite element spaces satisfying the inf-sup condition, e.g., \cite{MarCro-1987,ForFor-1989,BaSa-1992}.
Secondly, the equal-order approximations for the velocity, pressure and stress with stabilization terms added to the usual weak formulation, see for instance~\cite{FanTanPha-1999}.
And finally, the elastic viscous split stress (EVSS) method, e.g., \cite{GuFo-1995,RajaArmBr-1990,Yu-1997}, in which the stress is split into two parts, the elastic and the viscous part.
Scheme~\eqref{lin_scheme} is classified into the second approach.
Theorems~\ref{thm:error_estimates} and~\ref{thm:error_estimates_pressure} imply that our method for the {\rm Peterlin} viscoelastic model is indeed stable and convergent.
\end{remark}
\begin{lemma}[ \cite{NT-2015-JSC,RuiTab-2002} ]\label{lem:composite_func}
Under Hypothesis~\ref{hyp:w} and the condition~\eqref{cond:dt_w_Jacobian} the following inequalities hold for any $n \in \{0,\ldots,N_T\}$.
\begin{align*}
\n{0}{\gcko \circ X_1^n} & \leq (1 + \alpha_{40} |\wcko^n|_{1,\infty} \Delta t) \n{0}{\gcko}, && \forall \gcko \in L^2(\Omega)^s, \\
\n{0}{\gcko - \gcko \circ X_1^n} & \leq \alpha_{41} \|\wcko^n\|_{0,\infty} \Delta t \, |\gcko|_1, && \forall \gcko \in H^1(\Omega)^s,
\end{align*}
where $s = \DIM$ or $\DIM \times \DIM$.
\end{lemma}
\begin{proof}
We prove only the former estimate, and see the proof of~\cite[Lemma~6]{NT-2015-JSC} for the latter.
Let $n \in \{0,\ldots,N_T\}$ be fixed arbitrarily.
By changing the variable from~$x$ to~$y \defeq X_1^n(x)$, we have
\[
\|\gcko \circ X_1^n\|_0^2
= \int_\Omega \gcko \( X_1^n (x) \)^2\,dx
= \int_\Omega \gcko ( y )^2 \fz{1}{J^n} \, dy
\le \(1 + \alpha_{40} |\wcko^n|_{1,\infty} \Delta t \)^2 \|\gcko \|_0^2,
\]
where $J^n$ is the Jacobian ${\rm det}({\partial y}/{\partial x})$.
Here we have used the estimate,
\begin{align*}
\fz{1}{J^n}
\le \fz{1}{1-|1-J^n|}
\le 1 + 2 |1-J^n|
\le 1 + 2 \alpha_{40} |\wcko^n|_{1,\infty} \Delta t
\le (1 + \alpha_{40} |\wcko^n|_{1,\infty} \Delta t)^2,
\end{align*}
which is derived from Proposition~\ref{prop:RT_TU}-(ii) and $1/(1-s) \le 1+2s~(s \in [0, 1/2])$.
\end{proof}
We use the following simplified version of the discrete Gronwall inequality~\cite[Lemma~5.1]{HeyRan-1990}.
\begin{lemma}\label{lem:Gronwall}
Let $\alpha$ and $\beta$ be non-negative numbers, $\Delta t$ a positive number, and $\{x^n\}_{n\ge 0}$ and $\{y^n\}_{n\ge 1}$ non-negative sequences.
Suppose the inequality
\begin{align*}
x^m + \Delta t \sum_{n=1}^m y^n \leq \alpha \Delta t \sum_{n=0}^{m-1} x^n + \beta, \quad \forall m \ge 0,
\end{align*}
holds.
Then, it holds that
\begin{align*}
x^m + \Delta t \sum_{n=1}^m y^n \leq (1+\alpha \Delta t)^m \beta, \quad \forall m \ge 0.
\end{align*}
\end{lemma}
%
%
%
%
%
%
%
\subsection{Proof of Proposition~\ref{prop:existence_uniqueness}}
For each time step~$n$ scheme~\eqref{lin_scheme} can be rewritten as
\begin{subequations}\label{lin_scheme_rewritten}
\begin{align}
\Bigl( \fz{\au}{\Delta t}, \vcko_h \Bigr) + \nu a_u(\au, \vcko_h) + b(\vcko_h,\ap) + ((\atrC) \aoC, \nabla \vcko_h)
&= (\mathbf{g}_h^n,\vcko_h), \qquad\quad \forall \vcko_h \in V_h,
\label{lin_scheme_rewritten_velocity_pressure}\\
b( \au, q_h) - \mathcal{S}_h( \ap, q_h) &=0, \qquad\qquad\quad\quad  \forall q_h \in Q_h,
\label{lin_scheme_rewritten_divergence}\\
\Bigl(\frac{\aC}{\Delta t},\Dcko_h\Bigr) + \varepsilon a_{c}\(\aC,\Dcko_h\) - 2\((\nabla\au)\aoC,\Dcko_h\) + \( (\aotrC)^2 \aC, \Dcko_h \)
&= (\mathbf{G}_h^n,\Dcko_h), \qquad\ \forall \Dcko_h \in W_h,
\label{lin_scheme_rewritten_tensor}
\end{align}
\end{subequations}
where $\mathbf{g}_h^n \defeq (1/\Delta t) (\aou\circ X_1^n) + \af$ and $\mathbf{G}_h^n \defeq (1/\Delta t) (\aoC \circ X_1^n) + (\aotrC)\Icko + \aF$.
Selecting specific bases of $V_h$, $Q_h$ and $W_h$ and expanding $\au$, $\ap$ and $\aC$ in terms of the associated basis functions, we can derive the system of linear equations from~\eqref{lin_scheme_rewritten}.
The existence and uniqueness of the solution is equivalent to the invertibility of the coefficient matrix of the system, which is obtained by proving $(\au,\ap,\aC)=(\mathbf{0},0,\mathbf{0})$ below when $(\mathbf{g}_h^n, \mathbf{G}_h^n) = (\mathbf{0},\mathbf{0})$.
Substituting $(\au,-\ap,\frac{1}{2}(\atrC)\Icko)$ into $(\vcko_h,q_h,\Dcko_h)$ in~\eqref{lin_scheme_rewritten} and adding~\eqref{lin_scheme_rewritten_divergence} to~\eqref{lin_scheme_rewritten_velocity_pressure}, we have
\begin{subequations}\label{lin_scheme_rewritten_2nd}
\begin{align}
\fz{1}{\Delta t} \n{0}{\au}^2 + 2\nu \n{0}{\D{\au}}^2 + \delta_0|\ap|_h^2 + \((\atrC)\aoC,\nabla \au\) & = 0,
\label{lin_scheme_rewritten_2nd_velocity_pressure} \\
\frac{1}{2\Delta t} \n{0}{\atrC}^2 + \fz{\varepsilon}{2} \n{0}{\nabla\atrC}^2 - \({\rm tr}[(\nabla\au)\aoC],\atrC\) + \fz{1}{2} \n{0}{\aotrC \atrC}^2 &= 0.
\label{lin_scheme_rewritten_2nd_tensor}
\end{align}
\end{subequations}
By the identity
\begin{align*}
\((\atrC)\aoC,\nabla \au\) - \({\rm tr}[(\nabla\au)\aoC], \atrC\)=0,
\end{align*}
the sum of~\eqref{lin_scheme_rewritten_2nd_velocity_pressure} and~\eqref{lin_scheme_rewritten_2nd_tensor} yields
\begin{align*}
\fz{1}{\Delta t} \n{0}{\au}^2 + 2\nu \n{0}{\D{\au}}^2 + \delta_0|\ap|_h^2
+ \frac{1}{2\Delta t} \n{0}{\atrC}^2 + \fz{\varepsilon}{2} \n{0}{\nabla\atrC}^2
+ \fz{1}{2} \n{0}{\aotrC \atrC}^2 = 0.
\end{align*}
Hence, we have $(\au, \ap) = (\mathbf{0}, 0)$.
Substituting~$\aC$ into~$\Dcko_h$ in~\eqref{lin_scheme_rewritten_tensor} and noting that $\au=\mathbf{0}$, we obtain
\begin{align*}
\frac{1}{\Delta t} \n{0}{\aC}^2 + \varepsilon \n{0}{\nabla\aC}^2 + \n{0}{(\aotrC)\aC}^2 = 0,
\end{align*}
which implies $\aC=0$.
Thus, we get $(\au,\ap,\aC)=(\mathbf{0},0,\mathbf{0})$, which completes the proof.
\qed
%
\subsection{An estimate at each time step}
In this subsection we present a proposition which is employed in the proof of Theorem~\ref{thm:error_estimates}.
\par
Let $(\hat{\ucko}_h, \hat{p}_h, \hat{\Ccko}_h)(t) \defeq \Pi_h^{\rm SP} (\ucko, p, \Ccko)(t) \in V_h \times Q_h \times W_h$ for $t\in [0,T]$ and let
\begin{align*}
\ecko_h^n & \defeq \ucko_h^n-\hat{\ucko}_h^n,
&
\epsilon_h^n &\defeq p_h^n-\hat{p}_h^n,
&
\Ecko_h^n & \defeq \Ccko_h^n-\hat{\Ccko}_h^n,
&
\etacko (t) & \defeq (\ucko-\hat{\ucko}_h)(t),
&
\Xicko (t) & \defeq (\Ccko-\hat{\Ccko}_h)(t).
\end{align*}
Then, from~\eqref{lin_scheme}, \eqref{StokesPoisson_projection} and~\eqref{weak_formulation}, we have for $n\ge 1$
\begin{subequations}\label{eqns:error}
\begin{align}
\biggl( \fz{\ecko_h^n - \ecko_h^{n-1} \circ X_1^n}{\Delta t}, \vcko_h \biggr) + \mathcal{A}_h \bigl( (\ecko_h^n,\epsilon_h^n), (\vcko_h,q_h) \bigr) & = \lA \rcko_h^n, \vcko_h\rA, & \forall (\vcko_h, q_h) & \in V_h\times Q_h,
\label{eq:error_up} \\
\biggl( \fz{\Ecko_h^n - \Ecko_h^{n-1} \circ X_1^n}{\Delta t}, \vcko_h \biggr) + \varepsilon a_c (\Ecko_h^n, \Dcko_h) & = \lA \Rcko_h^n, \Dcko_h\rA,
& \forall \Dcko_h & \in W_h,
\label{eq:error_C}
\end{align}
\end{subequations}
where
\begin{align}
\rcko_h^n & \defeq \sum_{i=1}^4 \rcko_{hi}^n \in V_h^\prime, \qquad \Rcko_h^n \defeq \sum_{i=1}^{11} \Rcko_{hi}^n \in W_h^\prime,
\label{defs:r_R}\\
\lA \rcko_{h1}^n, \vcko_h \rA & \defeq \( \Dt{\ucko^n} - \fz{\ucko^n - \ucko^{n-1} \circ X_1^n}{\Delta t}, \vcko_h\), \notag\\
\lA \rcko_{h2}^n, \vcko_h \rA & \defeq \fz{1}{\Delta t}\( \etacko^n - \etacko^{n-1} \circ X_1^n, \vcko_h \), \notag\\
\lA \rcko_{h3}^n, \vcko_h \rA & \defeq \bigl( (\tr \Ccko^n)(\Ccko^n-\Ccko^{n-1} + \Xicko^{n-1}-\Ecko_h^{n-1}), \nabla \vcko_h \bigr), \notag\\
\lA \rcko_{h4}^n, \vcko_h \rA & \defeq \( [\tr (\Xicko^n-\Ecko_h^n)] \Ccko_h^{n-1}, \nabla \vcko_h \), \notag\\
\lA \Rcko_{h1}^n, \Dcko_h \rA & \defeq \( \Dt{\Ccko^n} - \fz{\Ccko^n - \Ccko^{n-1} \circ X_1^n}{\Delta t}, \Dcko_h\), \notag\\
\lA \Rcko_{h2}^n, \Dcko_h \rA & \defeq \fz{1}{\Delta t} \( \Xicko^n - \Xicko^{n-1} \circ X_1^n, \Dcko_h\), \notag\\
\lA \Rcko_{h3}^n, \Dcko_h \rA & \defeq - \varepsilon (\Xicko^n, \Dcko_h), \notag\\
\lA \Rcko_{h4}^n, \Dcko_h \rA & \defeq 2\( (\nabla \ecko_h^n) \Ccko_h^{n-1}, \Dcko_h\), \notag\\
\lA \Rcko_{h5}^n, \Dcko_h \rA & \defeq -2\( (\nabla\etacko^n) \Ccko_h^{n-1}, \Dcko_h\), \notag\\
\lA \Rcko_{h6}^n, \Dcko_h \rA & \defeq -2\( (\nabla \ucko^n) (\Ccko^n-\Ccko^{n-1} + \Xicko^{n-1}-\Ecko_h^{n-1}), \Dcko_h\), \notag\\
\lA \Rcko_{h7}^n, \Dcko_h \rA & \defeq \( (\tr \Ccko_h^{n-1})^2 (\Xicko^n - \Ecko_h^n), \Dcko_h\), \notag\\
\lA \Rcko_{h8}^n, \Dcko_h \rA & \defeq -\bigl( [\tr (\Ccko_h^{n-1}+\hat{\Ccko}_h^{n-1})] (\tr \Ecko_h^{n-1}) \Ccko^n, \Dcko_h\bigr), \notag\\
\lA \Rcko_{h9}^n, \Dcko_h \rA & \defeq \bigl( [\tr (\Ccko^{n-1}+\hat{\Ccko}_h^{n-1})] (\tr \Xicko^{n-1}) \Ccko^n, \Dcko_h \bigr), \notag\\
\lA \Rcko_{h10}^n, \Dcko_h \rA & \defeq \( [\tr (\Ccko^n+\Ccko^{n-1})] [\tr (\Ccko^n-\Ccko^{n-1})] \Ccko^n, \Dcko_h\), \notag\\
\lA \Rcko_{h11}^n, \Dcko_h \rA & \defeq -\( [\tr (\Ccko^n - \Ccko^{n-1} + \Xicko^{n-1} -\Ecko_h^{n-1})]\Icko, \Dcko_h\). \notag
\end{align}
We note that
\begin{align}
(\ecko_h^0, \Ecko_h^0) = (\ucko_h^0, \Ccko_h^0) - (\hat{\ucko}_h^0, \hat{\Ccko}_h^0) = [\Pi_h^{\rm SP} (0, -p^0, 0)]_{1,3}.
\label{eq:initial_approx_value}
\end{align}
\par
In the following we use the constants~$\alpha_i$ defined in Lemma~$i$, $i=1,\ldots,4$, and the notation~$\mathbb{H}^2 \defeq H^2(\Omega)^2\times H^1(\Omega)\times H^2(\Omega)^{2 \times 2}$.
%
\begin{prop}\label{prop:eh_epsh_Gronwall}
Suppose that Hypotheses~\ref{hyp:w} and~\ref{hyp:regularity} hold and assume~\eqref{cond:dt_w_Jacobian}.
Let $M_0 \ge 1$ be a positive constant independent of $h$ and $\Delta t$.
Let $(\ucko_h, p_h, \Ccko_h)$ be the solution of scheme~\eqref{lin_scheme} with~\eqref{cond:if}. Suppose that for an $n\in\{1,\ldots,N_T\}$
\begin{align}
\|\Ccko_h^{n-1}\|_{0,\infty} \le M_0.
\label{ieq:M}
\end{align}
Then, there exist positive constants~$c_1$ and~$c_2$, dependent on $M_0$, $\nu$ and $\varepsilon$ but independent of $h$ and $\Delta t$, such that
\begin{align}
& \ol{D}_{\Delta t} \Bigl( \fz{1}{2} \|\ecko_h^n\|_0^2 + \fz{1}{2} \|\Ecko_h^n\|_0^2 + \fz{\nu\varepsilon}{64 \alpha_1^2 d^2 M_0^2} |\Ecko_h^n|_1^2 \Bigr) + \fz{\nu}{ 2\alpha_1^2}\|\ecko_h^n\|_1^2 + \delta_0 |\epsilon_h^n|_h^2 + \fz{\nu}{64 \alpha_1^2 d^2 M_0^2} \|\ol{D}_{\Delta t}\Ecko_h^n\|_0^2 \notag\\
& \quad \le c_1 \Bigl( \fz{1}{2} \|\ecko_h^{n-1}\|_0^2 + \fz{1}{2} \|\Ecko_h^{n-1}\|_0^2 + \fz{\nu\varepsilon}{64\alpha_1^2 d^2 M_0^2} |\Ecko_h^{n-1}|_1^2 + \fz{1}{2} \|\Ecko_h^n\|_0^2 \Bigr) \notag\\
&\qquad + c_2 \Bigl[ \Delta t \|(\ucko,\Ccko)\|_{Z^2(t^{n-1},t^n)}^2
+ h^2 \Bigl( \fz{1}{\Delta t} \|(\ucko,p,\Ccko)\|_{H^1(t^{n-1},t^n; \mathbb{H}^2)}^2 +1 \Bigr) \Bigr].
\label{ieq:eh_epsh_Gronwall_H1}
\end{align}
\end{prop}
%
\par
For the proof we use the next lemma, which is proved in Appendix~\ref{subsec:proof_r_R}.
\begin{lemma}\label{lem:estimates_r_R}
Suppose Hypotheses~\ref{hyp:w} and~\ref{hyp:regularity} hold.
Let $n\in \{1,\ldots,N_T\}$ be any fixed number.
Then, under the condition~\eqref{cond:dt_w_Jacobian} it holds that
\begin{subequations}
\begin{align}
\| \rcko_{h1}^n \|_0 & \le c_w\sqrt{\Delta t} \|\ucko\|_{Z^2(t^{n-1},t^n)},
\label{ieq:r1}
\\
\| \rcko_{h2}^n \|_0 & \le \fz{c_w h}{\nu\sqrt{\Delta t}} \| (\ucko, p) \|_{H^1(t^{n-1},t^n; H^2\times H^1)},
\label{ieq:r2}
\\
\| \rcko_{h3}^n \|_{-1} & \le c_s \bigl( \|\Ecko_h^{n-1}\|_0 + \sqrt{\Delta t}\| \Ccko \|_{H^1(t^{n-1},t^n; L^2)} + h \bigr),
\label{ieq:r3}
\\
\| \rcko_{h4}^n \|_{-1} & \le c_s \|\Ccko_h^{n-1}\|_{0,\infty} \( \|\Ecko_h^n\|_0 + h \),
\label{ieq:r4}
\\
\| \Rcko_{h1}^n \|_0 & \le c_w \sqrt{\Delta t} \|\Ccko\|_{Z^2(t^{n-1},t^n)},
\label{ieq:R1}
\\
\| \Rcko_{h2}^n \|_0 & \le \fz{c_w h}{\sqrt{\Delta t}} \| \Ccko \|_{H^1(t^{n-1},t^n; H^2)},
\label{ieq:R2}
\\
\| \Rcko_{h3}^n \|_0 & \le c_s h,
\label{ieq:R3}
\\
\| \Rcko_{h4}^n \|_0 & \le 2d \|\Ccko_h^{n-1}\|_{0,\infty} \|\ecko_h^n\|_1,
\label{ieq:R4}
\\
\| \Rcko_{h5}^n \|_0 & \le c_s \|\Ccko_h^{n-1}\|_{0,\infty} h,
\label{ieq:R5}
\\
\| \Rcko_{h6}^n \|_0 & \le c_s \bigl( \|\Ecko_h^{n-1}\|_0 + \sqrt{\Delta t}\| \Ccko \|_{H^1(t^{n-1},t^n; L^2)} + h \bigr),
\label{ieq:R6}
\\
\| \Rcko_{h7}^n \|_0 & \le c_s \|\Ccko_h^{n-1}\|_{0,\infty}^2 ( \|\Ecko_h^n\|_0 + h),
\label{ieq:R7}
\\
\| \Rcko_{h8}^n \|_0 & \le c_s (\|\Ccko_h^{n-1}\|_{0,\infty}+1)\|\Ecko_h^{n-1}\|_0,
\label{ieq:R8}
\\
\| \Rcko_{h9}^n \|_0 & \le c_s h,
\label{ieq:R9}
\\
\| \Rcko_{h10}^n \|_0 & \le c_s \sqrt{\Delta t}\| \Ccko \|_{H^1(t^{n-1},t^n; L^2)},
\label{ieq:R10}
\\
\| \Rcko_{h11}^n \|_0 & \le c_s ( \|\Ecko_h^{n-1}\|_0 + \sqrt{\Delta t}\| \Ccko \|_{H^1(t^{n-1},t^n; L^2)} + h).
\label{ieq:R11}
\end{align}
\end{subequations}
\end{lemma}
\begin{proof}[Proof of Proposition~\ref{prop:eh_epsh_Gronwall}]
Substituting $(\ecko_h^n, -\epsilon_h^n)$ into $(\vcko_h, q_h)$ in~\eqref{eq:error_up} and noting that
\begin{align*}
\( \fz{\ecko_h^n - \ecko_h^{n-1} \circ X_1^n}{\Delta t}, \ecko_h^n \)
& \ge
\fz{1}{2\Delta t} \bigl( \|\ecko_h^n\|_0^2 - \|\ecko_h^{n-1} \circ X_1^n\|_0^2 \bigr)
\ge
\fz{1}{2\Delta t} \Bigl[ \|\ecko_h^n\|_0^2 - ( 1 + \alpha_{40} |\wcko^n|_{1,\infty} \Delta t )^2 \|\ecko_h^{n-1}\|_0^2 \Bigr] \\
& \ge \ol{D}_{\Delta t} \Bigl( \fz{1}{2} \|\ecko_h^n\|_0^2 \Bigr) - c_w \|\ecko_h^{n-1}\|_0^2, \\
\mathcal{A}_h\bigl( (\ecko_h^n, \epsilon_h^n), (\ecko_h^n, -\epsilon_h^n) \bigr) & \ge \fz{2\nu}{\alpha_1^2} \|\ecko_h^n\|_1^2 + \delta_0 |p_h^n|_h^2,\\
\lA \rcko_h^n, \ecko_h^n \rA & \le \|\rcko_h^n \|_{-1} \|\ecko_h^n\|_1
\le \fz{\alpha_1^2}{4\nu} \| \rcko_h^n \|_{-1}^2 + \fz{\nu}{ \alpha_1^2}\|\ecko_h^n\|_1^2,
\end{align*}
we have
\begin{align}
& \ol{D}_{\Delta t} \Bigl( \fz{1}{2} \|\ecko_h^n\|_0^2 \Bigr) + \fz{\nu}{ \alpha_1^2}\|\ecko_h^n\|_1^2 + \delta_0 |\epsilon_h^n|_h^2
\le \fz{ \alpha_1^2}{4\nu} \| \rcko_h^n \|_{-1}^2 + c_w \|\ecko_h^{n-1}\|_0^2.
\label{ieq:error_up_substitution}
\end{align}
Similarly, substituting $\Ecko_h^n$ and $\ol{D}_{\Delta t} \Ecko_h^n$ into $\Dcko_h$ in~\eqref{eq:error_C} and noting that
\begin{align*}
\( \fz{\Ecko_h^n - \Ecko_h^{n-1} \circ X_1^n}{\Delta t}, \Ecko_h^n\) & \ge \ol{D}_{\Delta t} \Bigl( \fz{1}{2} \|\Ecko_h^n\|_0^2 \Bigr) - c_w \|\Ecko_h^{n-1}\|_0^2,\\
\varepsilon a_c ( \Ecko_h^n, \Ecko_h^n ) & = \varepsilon |\Ecko_h^n|_1^2 \ge 0,\\
\lA \Rcko_h^n, \Ecko_h^n \rA
& \le \|\Rcko_h^n\|_0 \|\Ecko_h^n\|_0 
\le \sum_{i\in \{ 1,\ldots,11 \}\setminus \{4\}} \|\Rcko_{hi}^n\|_0 \|\Ecko_h^n\|_0 + \|\Rcko_{h4}^n\|_0 \|\Ecko_h^n\|_0 \\
& \le \sum_{i\in \{ 1,\ldots, 11 \}\setminus \{4\}} \Bigl( \fz{5}{2} \|\Rcko_{hi}^n\|_0^2 + \fz{1}{10} \|\Ecko_h^n\|_0^2 \Bigr) + 2 d M_0 \|\ecko_h^n\|_1 \|\Ecko_h^n\|_0 \quad \mbox{(by~\eqref{ieq:R4},\eqref{ieq:M})} \\
& \le \fz{5}{2} \sum_{i\in \{ 1,\ldots, 11 \}\setminus \{4\}} \|\Rcko_{hi}^n\|_0^2 + \|\Ecko_h^n\|_0^2 + \fz{\nu}{4\alpha_1^2} \|\ecko_h^n\|_1^2 + \fz{4 \alpha_1^2 d^2 M_0^2}{\nu}\|\Ecko_h^n\|_0^2 \\
& = \fz{5}{2} \sum_{i\in \{ 1,\ldots, 11 \}\setminus \{4\}} \|\Rcko_{hi}^n\|_0^2 + \Bigl( 1 + \fz{4 \alpha_1^2 d^2 M_0^2}{\nu} \Bigr) \|\Ecko_h^n\|_0^2 + \fz{\nu}{4\alpha_1^2} \|\ecko_h^n\|_1^2, \\
\biggl( \fz{\Ecko_h^n - \Ecko_h^{n-1} \circ X_1^n}{\Delta t}, \ol{D}_{\Delta t} \Ecko_h^n \biggr)
& = \biggl( \ol{D}_{\Delta t} \Ecko_h^n + \fz{\Ecko_h^{n-1} - \Ecko_h^{n-1} \circ X_1^n}{\Delta t}, \ol{D}_{\Delta t} \Ecko_h^n \biggr) \\
& \ge \|\ol{D}_{\Delta t} \Ecko_h^n\|_0^2 - \alpha_{41} \|\wcko^n\|_{0,\infty} |\Ecko_h^{n-1}|_1 \|\ol{D}_{\Delta t} \Ecko_h^n\|_0, \\
& \ge \|\ol{D}_{\Delta t} \Ecko_h^n\|_0^2 - c_w |\Ecko_h^{n-1}|_1^2 - \fz{1}{4} \|\ol{D}_{\Delta t} \Ecko_h^n\|_0^2, \\
& = \fz{3}{4}\|\ol{D}_{\Delta t} \Ecko_h^n\|_0^2 - c_w |\Ecko_h^{n-1}|_1^2, \\
\varepsilon a_c \bigl( \Ecko_h^n, \ol{D}_{\Delta t} \Ecko_h^n \bigr)
& \ge \ol{D}_{\Delta t} \Bigl( \fz{\varepsilon}{2} |\Ecko_h^n|_1^2 \Bigr), \\
\bigl\lA \Rcko_h^n, \ol{D}_{\Delta t} \Ecko_h^n \bigr\rA
& \le \|\Rcko_h^n\|_0 \|\ol{D}_{\Delta t} \Ecko_h^n\|_0 
\le \sum_{i\in \{ 1,\ldots, 11 \}\setminus \{4\}} \|\Rcko_{hi}^n\|_0 \|\ol{D}_{\Delta t} \Ecko_h^n\|_0 + \|\Rcko_{h4}^n\|_0 \|\ol{D}_{\Delta t} \Ecko_h^n\|_0 \\
& \le \sum_{i\in \{ 1,\ldots, 11 \}\setminus \{4\}} \Bigl( 20 \|\Rcko_{hi}^n\|_0^2 + \fz{1}{80} \|\ol{D}_{\Delta t} \Ecko_h^n\|_0^2 \Bigr) + 2 d M_0 \|\ecko_h^n\|_1 \|\ol{D}_{\Delta t} \Ecko_h^n\|_0 \quad \mbox{(by~\eqref{ieq:R4},\eqref{ieq:M})} \\
& \le 20 \sum_{i\in \{ 1,\ldots, 11 \}\setminus \{4\}} \|\Rcko_{hi}^n\|_0^2 + \fz{1}{8} \|\ol{D}_{\Delta t} \Ecko_h^n\|_0^2 + 8 d^2 M_0^2 \|\ecko_h^n\|_1^2 + \fz{1}{8} \|\ol{D}_{\Delta t} \Ecko_h^n\|_0^2 \\
& = 20 \sum_{i\in \{ 1,\ldots, 11 \}\setminus \{4\}} \|\Rcko_{hi}^n\|_0^2 + \fz{1}{4} \|\ol{D}_{\Delta t} \Ecko_h^n\|_0^2 + 8 d^2 M_0^2 \|\ecko_h^n\|_1^2,
\end{align*}
we have the following two inequalities,
\begin{subequations}
\begin{align}
& \ol{D}_{\Delta t} \Bigl( \fz{1}{2} \|\Ecko_h^n\|_0^2 \Bigr) \le \fz{5}{2} \sum_{i\in \{ 1,\ldots, 11 \}\setminus \{4\}} \|\Rcko_{hi}^n\|_0^2 + \Bigl( 1 + \fz{4 \alpha_1^2 d^2 M_0^2}{\nu} \Bigr) \|\Ecko_h^n\|_0^2 + c_w \|\Ecko_h^{n-1}\|_0^2 + \fz{\nu}{4\alpha_1^2} \|\ecko_h^n\|_1^2,
\label{ieq:error_C_substitution}\\
& \ol{D}_{\Delta t} \Bigl( \fz{\varepsilon}{2} |\Ecko_h^n|_1^2 \Bigr) + \fz{1}{2} \|\ol{D}_{\Delta t} \Ecko_h^n\|_0^2 \le 20 \sum_{i\in \{ 1,\ldots, 11 \}\setminus \{4\}} \|\Rcko_{hi}^n\|_0^2 + c_w |\Ecko_h^{n-1}|_1^2 + 8 d^2 M_0^2 \|\ecko_h^n\|_1^2.
\label{ieq:error_C_substitution_D_dt_E}
\end{align}
\end{subequations}
Lemma~\ref{lem:estimates_r_R}, \eqref{defs:r_R} and~\eqref{ieq:M} imply that
\begin{subequations}\label{ieqs:r_R_with_M}
\begin{align}
\|\rcko_h^n\|_{-1}^2
& \le c_{w,s} \bigl( M_0^2 \|\Ecko_h^n\|_0^2 + \|\Ecko_h^{n-1}\|_0^2 \bigr) \notag\\
& \qquad + \fz{c_{w,s}^\prime}{\nu} \Bigl[ \Delta t \|(\ucko,\Ccko)\|_{Z^2(t^{n-1},t^n)}^2
+ h^2 \Bigl( \fz{1}{\Delta t} \|(\ucko,p)\|_{H^1(t^{n-1},t^n; H^2\times H^1)}^2 +M_0^2 + 1 \Bigr) \Bigr],
\\
\sum_{i\in \{ 1,\ldots, 11 \}\setminus \{4\}} \|\Rcko_{hi}^n\|_0^2
& \le c_{w,s} \Bigl[ M_0^4 \|\Ecko_h^n\|_0^2 + (M_0^2+1) \|\Ecko_h^{n-1}\|_0^2 \Bigr] \notag\\
& \qquad + c_{w,s}^\prime \Bigl[ \Delta t \|\Ccko\|_{Z^2(t^{n-1},t^n)}^2 + h^2 \Bigl( \fz{1}{\Delta t} \|\Ccko\|_{H^1(t^{n-1},t^n; H^2)}^2 + M_0^4 + M_0^2 + 1 \Bigr) \Bigr].
\end{align}
\end{subequations}
Multiplying~\eqref{ieq:error_C_substitution_D_dt_E} by~$\nu / (32 \alpha_1^2 d^2 M_0^2)$, adding it and~\eqref{ieq:error_C_substitution} to~\eqref{ieq:error_up_substitution} and using~\eqref{ieqs:r_R_with_M}, we get
\begin{align*}
& \ol{D}_{\Delta t} \Bigl( \fz{1}{2} \|\ecko_h^n\|_0^2 + \fz{1}{2} \|\Ecko_h^n\|_0^2 + \fz{\nu\varepsilon}{64 \alpha_1^2 d^2 M_0^2} |\Ecko_h^n|_1^2 \Bigr) + \fz{\nu}{ 2\alpha_1^2}\|\ecko_h^n\|_1^2 + \delta_0 |\epsilon_h^n|_h^2 + \fz{\nu}{64 \alpha_1^2 d^2 M_0^2} \|\ol{D}_{\Delta t}\Ecko_h^n\|_0^2 \\
& \le p_1(M_0) \Bigl( \fz{1}{2} \|\ecko_h^{n-1}\|_0^2 + \fz{1}{2} \|\Ecko_h^{n-1}\|_0^2 + \fz{\nu\varepsilon}{64 \alpha_1^2 d^2 M_0^2} |\Ecko_h^{n-1}|_1^2 + \fz{1}{2} \|\Ecko_h^n\|_0^2 \Bigr) \\
&\quad + p_2(M_0) \Bigl[ \Delta t \|(\ucko,\Ccko)\|_{Z^2(t^{n-1},t^n)}^2
+ h^2 \Bigl( \fz{1}{\Delta t} \|(\ucko,p,\Ccko)\|_{H^1(t^{n-1},t^n; \mathbb{H}^2)}^2 +1 \Bigr) \Bigr],
\end{align*}
where $p_1(\xi) = p_1(\xi; \nu, \varepsilon)$ and $p_2(\xi) = p_2(\xi; \nu)$ are polynomials in~$\xi$ defined by
\begin{align}
p_1 : \quad & c_{w,s} \Bigl[ \fz{1}{\nu} (\xi^2 + 1) + \Bigl( 1 + \fz{\nu}{\xi^2} \Bigr) (\xi^4 + \xi^2 + 1) + \Bigl(1+\fz{\xi^2}{\nu} \Bigr) + \fz{1}{\varepsilon} \Bigr] 
\le \fz{c_{w,s}}{\nu\varepsilon} ( \xi^4 + 4\xi^2 + 6 )
\eqdef p_1(\xi; \nu, \varepsilon),
\label{def:p1} \\
p_2 : \quad & c_{w,s} \Bigl[ \fz{1}{\nu^2} ( \xi^2 + 1 ) + \Bigl( 1 + \fz{\nu}{\xi^2} \Bigr) (\xi^4 + \xi^2 + 1) \Bigr] 
\le \fz{c_{w,s}}{\nu^2} (\xi^4 + 3\xi^2 + 4) 
\eqdef p_2 (\xi; \nu). \notag
\end{align}
In the inequalities above the assumptions $\nu, \varepsilon \in (0, 1]$ and $M_0 \ge 1$ have been employed.
By taking $c_1 = p_1(M_0; \nu, \varepsilon)$ and $c_2 = p_2(M_0; \nu)$ we finally obtain~\eqref{ieq:eh_epsh_Gronwall_H1}.
\end{proof}
%
%
%
%
\subsection{Proof of Theorem~\ref{thm:error_estimates}}
We prove Theorem~\ref{thm:error_estimates} through three steps, where the function~$D(h)$ defined in~\eqref{Dofh} is often used.
\medskip\\
\textit{Step~1} (Setting $c_0$ and $h_0$):\
From~\eqref{cond:if} and~\eqref{eq:initial_approx_value}
we have
\begin{align}
\|\ecko_h^0\|_0
& \le \|\ucko_h^0 - \ucko^0\|_1 +\| \ucko^0 - \hat{\ucko}_h^0\|_1 \le 2 \fz{\alpha_{31}}{\nu} h \|(u, p)^0\|_{H^2 \times H^1} = \sqrt{2} c_I h
\label{ieq:c_I}
\end{align}
for $c_I \defeq (\sqrt{2} \alpha_{31}/\nu) \|(u, p)^0\|_{H^2 \times H^1}$.
The constants~$c_1$ and $c_2$ in Proposition~\ref{prop:eh_epsh_Gronwall} depend on $M_0$.
Now, we take $M_0 = \|\Ccko\|_{C(L^\infty)}+1$.
Then, $c_1$ and $c_2$ are fixed.
Let $c_3$ and $c_\ast$ be constants defined by
\begin{align}
c_3 \defeq
\exp\Bigl(\fz{3c_1T}{2}\Bigr) \max \Bigl\{ \sqrt{c_2} \|(\ucko,\Ccko)\|_{Z^2},  \sqrt{c_2} \bigl( \|(\ucko,p,\Ccko)\|_{H^1(\mathbb{H}^2)} + \sqrt{T} \bigr) + c_I\Bigr\}.
\label{def:c3}
\end{align}
and $c_\ast \defeq  c_3 \, (8 \alpha_1 d M_0 / \sqrt{\nu\varepsilon}\,)$.
We can choose sufficiently small positive constants $c_0$ and $h_0$ such that
\begin{subequations}\label{def:c0_h0}
\begin{align}
\alpha_{21} \bigl[ c_\ast \{ c_0 + h_0 D(h_0) \} + (\alpha_{20} + \alpha_{32}) h_0 D(h_0) \|\Ccko\|_{C(H^2)} \bigr] & \le 1,
\label{def:c0_h0_Linf}
\\
( \Delta t\le ) \ \ \  \fz{c_0}{D(h_0)} & \le \fz{1}{2c_1},
\label{def:c0_h0_Gronwall}\\
(\Delta t |\wcko|_{1,\infty}\le)\ \ \  \fz{c_0 |\wcko|_{1,\infty}}{D(h_0)} & \le \fz{1}{4},
\label{def:c0_h0_Jacobian}
\end{align}
\end{subequations}
since $h D(h)$ and~$1/D(h)$ tend to zero as $h$ tends to zero.
\par
Let $(h, \Delta t)$ be any pair satisfying~\eqref{condition:h_dt}.
Since condition \eqref{cond:dt_w_bijective} is satisfied, Proposition~\ref{prop:existence_uniqueness} ensures the existence and uniqueness of the solution~$(\ucko_h, p_h,$ $\Ccko_h) = \{ (\ucko_h^n, p_h^n, \Ccko_h^n) \}_{n=1}^{N_T} \subset V_h\times Q_h\times W_h$ of scheme~\eqref{lin_scheme} with \eqref{cond:if}.
\medskip\\
\textit{Step~2} (Induction):\
By induction we show that the following property P($n$) holds for $n\in\{0,\ldots,N_T\}$,
\begin{align*}
\mbox{P($n$):}
\left\{
\begin{aligned}
&
\begin{aligned}
&
{\rm (a)}~\fz{1}{2} \|\ecko_h^n\|_0^2 + \fz{1}{2} \|\Ecko_h^n\|_0^2 + \fz{\nu\varepsilon}{64 \alpha_1^2 d^2 M_0^2} |\Ecko_h^n|_1^2 
+\fz{\nu}{ 2\alpha_1^2}\|\ecko_h\|_{\ell^2_n(H^1)}^2
+\delta_0 |\epsilon_h|_{\ell^2_n(|\cdot|_h)}^2
+\fz{\nu}{64 \alpha_1^2 d^2 M_0^2} \|\ol{D}_{\Delta t}\Ecko_h\|_{\ell^2_n(L^2)}^2 \\
&
\quad \  \le \exp (3c_1 n\Delta t)
\Bigl[ \fz{1}{2}\|\ecko_h^0\|_0^2 + \fz{1}{2} \|\Ecko_h^0\|_0^2 + \fz{\nu\varepsilon}{64 \alpha_1^2 d^2 M_0^2} |\Ecko_h^0|_1^2 \\
& \qquad\qquad\qquad\qquad\qquad + c_2\Bigl\{ \Delta t^2 \|(\ucko,\Ccko)\|_{Z^2(0,t^n)}^2
+ h^2 \bigl( \|(\ucko,p,\Ccko)\|_{H^1(0,t^n; \mathbb{H}^2)}^2 + n\Delta t \bigr)
\Bigr\} \Bigr],
\end{aligned}
\\
& {\rm (b)}~\|\Ccko_h^n\|_{0,\infty} \le \|\Ccko\|_{C(L^\infty)}+1,
\end{aligned}
\right.
\end{align*}
where $\|\ecko_h\|_{\ell^2_n(H^1)} = |\epsilon_h|_{\ell^2_n(|\cdot|_h)} = \|\ol{D}_{\Delta t}\Ecko_h\|_{\ell^2_n(L^2)} = 0$ for $n=0$.
\par
P($n$)-(a) can be rewritten as
\begin{align}
x_n + \Delta t\sum_{i=1}^n y_i \le \exp( 3c_1n\Delta t ) \Bigl(x_0 + \Delta t \sum_{i=1}^n b_i \Bigr),
\label{ieq:proof_thm_P_n}
\end{align}
where
\begin{align*}
x_n & \defeq \fz{1}{2} \|\ecko_h^n\|_0^2 + \fz{1}{2} \|\Ecko_h^n\|_0^2 + \fz{\nu\varepsilon}{64 \alpha_1^2 d^2 M_0^2} |\Ecko_h^n|_1^2,
\qquad
y_i \defeq \fz{\nu}{2\alpha_1^2}\|\ecko_h^i\|_1^2 + \delta_0|\epsilon_h^i|_h^2 + \fz{\nu}{64 \alpha_1^2 d^2 M_0^2} \|\ol{D}_{\Delta t}\Ecko_h^i\|_0^2,\\
b_i & \defeq c_2\Bigl\{ \Delta t \|(\ucko,\Ccko)\|_{Z^2(t^{i-1},t^i)}^2 + h^2 \Bigl( \fz{1}{\Delta t} \|(\ucko,p,\Ccko)\|_{H^1(t^{i-1},t^i; \mathbb{H}^2)}^2 +1 \Bigr) \Bigr\}.
\end{align*}
\par
We firstly prove the general step in the induction.
Supposing that P($n-1$) holds true for an integer $n\in\{1,\ldots,N_T\}$, we prove that P($n$) also holds.
We prove P($n$)-(a).
Since \eqref{cond:dt_w_Jacobian} and~\eqref{ieq:M} with~$M_0 = \|\Ccko\|_{C(L^\infty)}+1~(\ge 1)$ are satisfied from~\eqref{def:c0_h0_Jacobian} and P($n-1$)-(b), respectively, we have~\eqref{ieq:eh_epsh_Gronwall_H1} from Proposition~\ref{prop:eh_epsh_Gronwall}.
The inequality~\eqref{ieq:eh_epsh_Gronwall_H1} implies that
\begin{align*}
\ol{D}_{\Delta t} x_n + y_n \le c_1 (x_n+x_{n-1}) + b_n,
\end{align*}
which leads to
\begin{align}
x_n + \Delta t y_n \le \exp (3c_1 \Delta t ) (x_{n-1} + \Delta t b_n)
\label{proof_thm1_1}
\end{align}
by $(1+c_1\Delta t)/(1-c_1\Delta t) \le (1+c_1\Delta t)(1+2c_1\Delta t) \le \exp (3c_1 \Delta t)$, where $c_1\Delta t \le 1/2$ from~\eqref{def:c0_h0_Gronwall}.
From~\eqref{proof_thm1_1} and P($n-1$)-(a)
we have
\begin{align*}
x_n + \Delta t\sum_{i=1}^n y_i
& \le \exp (3c_1\Delta t) (x_{n-1} + \Delta t b_n) + \Delta t\sum_{i=1}^{n-1} y_i
\le \exp (3c_1\Delta t) \biggl( x_{n-1} + \Delta t\sum_{i=1}^{n-1} y_i + \Delta t b_n \biggr) \\
& \le \exp (3c_1\Delta t) \biggl[ \exp\bigl\{ 3c_1(n-1)\Delta t\bigr\} \biggl(x_0 + \Delta t \sum_{i=1}^{n-1}b_i \biggr) + \Delta t b_n \biggr]\\
& \le \exp (3c_1n\Delta t) \biggl(x_0 + \Delta t \sum_{i=1}^n b_i \biggr).
\end{align*}
Thus, we obtain~P($n$)-(a).
\par
For the proof of P($n$)-(b) we prepare the estimate of $\|\Ecko_h^n\|_1$.
We have
\begin{align}
x_0 = \fz{1}{2} \|\ecko_h^0\|_0^2 + \fz{1}{2} \|\Ecko_h^0\|_0^2 + \fz{\nu\varepsilon}{64 \alpha_1^2 d^2 M_0^2} |\Ecko_h^0|_1^2 = \fz{1}{2}\|\ecko_h^0\|_0^2 
\le c_I^2 h^2
\label{ieq:eh0_epsh0}
\end{align}
from~\eqref{ieq:c_I}.
P($n$)-(a) with~\eqref{ieq:eh0_epsh0} implies that
\begin{align}
&
\fz{1}{2} \|\ecko_h^n\|_0^2 + \fz{1}{2} \|\Ecko_h^n\|_0^2 + \fz{\nu\varepsilon}{64 \alpha_1^2 d^2 M_0^2} |\Ecko_h^n|_1^2
+\fz{\nu}{ 2\alpha_1^2}\|\ecko_h\|_{\ell^2_n(H^1)}^2
+\delta_0 |\epsilon_h|_{\ell^2_n(|\cdot|_h)}^2
+ \fz{\nu}{64 \alpha_1^2 d^2 M_0^2} \|\ol{D}_{\Delta t}\Ecko_h\|_{\ell^2_n(L^2)}^2
\notag \\
&\quad \le \exp (3c_1 T)
\Bigl[ c_I^2h^2
+ c_2\Bigl\{ \Delta t^2 \|(\ucko,\Ccko)\|_{Z^2}^2
+ h^2 \bigl( \|(\ucko,p,\Ccko)\|_{H^1(\mathbb{H}^2)}^2 + T \bigr)
\Bigr\} \Bigr] \notag\\
&\quad \le \exp (3c_1 T)
\Bigl[ c_2 \Delta t^2 \|(\ucko,\Ccko)\|_{Z^2}^2 + h^2 \Bigl\{ c_2 \bigl( \|(\ucko,p,\Ccko)\|_{H^1(\mathbb{H}^2)}^2 + T \bigr) + c_I^2 \Bigr\} \Bigr]
\notag\\
&\quad \le \bigl\{ c_3 (\Delta t +h) \bigr\}^2,
\label{ieq:c3}
\end{align}
which yields
\begin{align}
\|\Ecko_h^n\|_1 & \le \fz{8\alpha_1dM_0}{\sqrt{\nu\varepsilon}} c_3 (\Delta t + h) = c_\ast (\Delta t + h)
\label{ieq:Ehn_H1}
\end{align}
from $\nu\varepsilon/(64\alpha_1^2d^2M_0^2) \le 1/(64d^2) < 1/2$.
\par
We prove P($n$)-(b) as follows:
\begin{align}
\|\Ccko_h^n\|_{0,\infty}
&\le \|\Ccko_h^n-\Pi_h\Ccko^n\|_{0,\infty} + \|\Pi_h\Ccko^n\|_{0,\infty}
\le \alpha_{21} D(h) \|\Ccko_h^n-\Pi_h\Ccko^n\|_1 + \|\Pi_h\Ccko^n\|_{0,\infty} \notag\\
&\le \alpha_{21} D(h) \bigl( \|\Ccko_h^n-\hat{\Ccko}_h^n\|_1 + \|\hat{\Ccko}_h^n-\Ccko^n\|_1 + \|\Ccko^n-\Pi_h\Ccko^n\|_1 \bigr) + \|\Pi_h\Ccko^n\|_{0,\infty} \notag\\
&\le \alpha_{21} D(h) \bigl[ c_\ast (\Delta t+h) + \alpha_{32} h\|\Ccko^n\|_2 + \alpha_{20} h \|\Ccko^n\|_2 \bigr] + \|\Ccko^n\|_{0,\infty} \notag\\
&\le \alpha_{21} \bigl[ c_\ast \{ c_0 + h_0 D(h_0) \} + (\alpha_{20} + \alpha_{32}) h_0 D(h_0) \|\Ccko\|_{C(H^2)} \bigr] + \|\Ccko\|_{C(L^\infty)} \notag\\
&\le 1 + \|\Ccko\|_{C(L^\infty)},
\label{ieq:Ch_0_infty}
\end{align}
from~\eqref{ieq:Ehn_H1}, \eqref{condition:h_dt} and~\eqref{def:c0_h0_Linf}.
Therefore, P($n$) holds true.
\par
The proof of P($0$) is easier than that of the general step.
P($0$)-(a) obviously holds with equality.
P($0$)-(b) is obtained as follows:
\begin{align*}
\|\Ccko_h^0\|_{0,\infty} & \le \|\Ccko_h^0-\Pi_h\Ccko^0\|_{0,\infty} + \|\Pi_h\Ccko^0\|_{0,\infty}
\le \alpha_{21} D(h) ( \|\Ccko_h^0 - \Ccko^0 \|_1 + \|\Ccko^0 - \Pi_h\Ccko^0\|_1 ) + \|\Pi_h\Ccko^0\|_{0,\infty} \\
& \le \alpha_{21} (\alpha_{20}+\alpha_{32}) h D(h) \|\Ccko^0 \|_2 + \|\Ccko^0\|_{0,\infty} \\
& \le 1 + \|\Ccko\|_{C(L^\infty)}.
\end{align*}
Thus, the induction is completed.
\medskip\\
\textit{Step~3}:\
Finally we derive~\eqref{ieq:stability} and~\eqref{ieq:main_results}.
Since P($N_T$) holds true, we have~\eqref{ieq:stability} and
\begin{align}
\|\ecko_h\|_{\ell^\infty(L^2)\cap\ell^2(H^1)}, \  \
|\epsilon_h|_{\ell^2(|\cdot|_h)}, \  \
\|\ol{D}_{\Delta t}\Ecko_h\|_{\ell^2(L^2)} \le
c_{\nu,\varepsilon} c_{w,s} (\Delta t + h)
\label{ieq:bar_eh_H1_Deh_L2}
\end{align}
from~\eqref{ieq:c3}.
Combining \eqref{ieq:bar_eh_H1_Deh_L2} and the estimates
\begin{align*}
\|\ucko_h-\ucko\|_{\ell^\infty(L^2)} & \le \|\ecko_h\|_{\ell^\infty(L^2)} + \|\etacko\|_{\ell^\infty(L^2)} \le \|\ecko_h\|_{\ell^\infty(L^2)} + \fz{\alpha_{31}}{\nu} h \|(\ucko, p)\|_{C(H^2\times H^1)}, \\
\Bigl\| \ol{D}_{\Delta t}\Ccko_h^n - \prz{\Ccko^n}{t} \Bigr\|_0
& \le \|\ol{D}_{\Delta t}\Ecko_h^n\|_0 + \|\ol{D}_{\Delta t}\Xicko^n\|_0 + \Bigl\| \ol{D}_{\Delta t}\Ccko^n - \prz{\Ccko^n}{t} \Bigr\|_0 \notag\\
& \le \|\ol{D}_{\Delta t}\Ecko_h^n\|_0 + \fz{\alpha_{32} h}{\sqrt{\Delta t}}\|\Ccko\|_{H^1(t^{n-1},t^n; H^2)} + \sqrt{\fz{\Delta t}{3}} \Bigl\| \prz{^2\Ccko}{t^2} \Bigr\|_{L^2(t^{n-1},t^n; L^2)} ,
\end{align*}
we can obtain the first and the last inequalities of~\eqref{ieq:main_results} with a positive constant $c_\dagger$ independent of $h$ and $\Delta t$.
The other inequalities of~\eqref{ieq:main_results} are similarly proved by using~\eqref{ieq:Ehn_H1} and~\eqref{ieq:bar_eh_H1_Deh_L2}.
\qed
\begin{remark}
We note that the error constant behaves like ${\cal O}(\exp[cT/(\nu\varepsilon)])$ $(\nu, \varepsilon \downarrow 0)$ with respect to the viscosity $\nu$ and the elastic diffusion coefficient~$\varepsilon$, since the main contribution is the exponential part of $c_3$ in~\eqref{def:c3}, i.e., $\exp[3c_1T/2] = \exp[3 p_1(\|\Ccko\|_{C(L^\infty)} + 1; \nu, \varepsilon) T/2] = \mathcal{O}(\exp[cT/(\nu\varepsilon)])$, where \eqref{def:p1} is used for the last equality.
Although the dependency on~$\nu$ and~$\varepsilon$ of the coefficient is not good, it seems hard to avoid it.
Similar coefficient $\mathcal{O}(\exp[cT/\nu])$ appears in the estimate of the {\rm Navier}--{\rm Stokes} equations, \cite{Suli-1988,BouMadMetRaz-1997}.
As for the estimate independent of~$\nu$, we refer to~\cite{OlsReu-2003} for the {\rm Stokes} equations and to~\cite{deFruGar-2016} for the {\rm Oseen} equations.  
\end{remark}
%
%
%
%
\subsection{A lemma for the proof of Theorem~\ref{thm:error_estimates_pressure}}
In the proof of Theorem~\ref{thm:error_estimates_pressure} we use the next lemma.
\begin{lemma}\label{lem:rh_pressure}
Suppose that Hypotheses~\ref{hyp:w} and~\ref{hyp:regularity} and the inequalities~\eqref{ieq:stability} and~\eqref{ieq:main_results} hold.
Let $m\in \{1,\ldots,N_T\}$ be any fixed number.
Then, under the condition~\eqref{cond:dt_w_Jacobian} we have the following.
\begin{subequations}
\begin{align}
\Delta t \sum_{n=1}^m \lA \rcko_{h1}^n, \ol{D}_{\Delta t}\ecko_h^n \rA, \ \ \Delta t \sum_{n=1}^m \lA \rcko_{h2}^n, \ol{D}_{\Delta t}\ecko_h^n \rA
& \le \fz{\Delta t}{6} \sum_{n=1}^m \| \ol{D}_{\Delta t} \ecko_h^n \|_0^2 + c_{\nu,\varepsilon} c_{w,s} ( \Delta t^2 + h^2 ), 
\label{ieq:r1_r2_pressure} \\
\Delta t \sum_{n=1}^m \lA \rcko_{h3}^n, \ol{D}_{\Delta t}\ecko_h^n \rA, \ \  \Delta t \sum_{n=1}^m \lA \rcko_{h4}^n, \ol{D}_{\Delta t}\ecko_h^n \rA
& \le \fz{\nu}{4} \|\D{\ecko_h^m}\|_0^2 + c_{\nu,\varepsilon} c_{w,s} ( \Delta t^2 + h^2 ).
\label{ieq:r3_r4_pressure}
\end{align}
\end{subequations}
\end{lemma}
\begin{proof}
The inequalities~\eqref{ieq:r1_r2_pressure} are obtained by combining~\eqref{ieq:r1} and~\eqref{ieq:r2} with
\begin{align*}
\lA \rcko_{hi}^n, \ol{D}_{\Delta t}\ecko_h^n \rA 
& \le \| \rcko_{hi}^n \|_0 \| \ol{D}_{\Delta t}\ecko_h^n \|_0
\le \fz{3}{2} \| \rcko_{hi}^n \|_0^2 + \fz{1}{6}\|\ol{D}_{\Delta t}\ecko_h^n\|_0^2, \quad i=1, 2.
\end{align*}
\par
We prove~\eqref{ieq:r3_r4_pressure}.
For $i=3, 4$ we have
\begin{align}
& \Delta t \sum_{n=1}^m \lA \rcko_{hi}^n, \ol{D}_{\Delta t}\ecko_h^n \rA 
= \sum_{n=1}^m ( \rcko_{hi}^n, \nabla\ecko_h^n - \nabla\ecko_h^{n-1} ) 
= ( \rcko_{hi}^m, \nabla\ecko_h^m ) - \sum_{n=1}^{m-1} ( \rcko_{hi}^{n+1} - \rcko_{hi}^n, \nabla\ecko_h^n ) - ( \rcko_{hi}^1, \nabla\ecko_h^0 ) \notag\\
& \quad \le \alpha_1 \|\rcko_{hi}^m\|_{-1} \|\D{\ecko_h^m}\|_0 + \sum_{n=1}^{m-1} \| \rcko_{hi}^{n+1} - \rcko_{hi}^n\|_0 \|\ecko_h^n\|_1 + \|\rcko_{hi}^1\|_{-1} \|\ecko_h^0\|_1 \notag\\
& \quad \le \fz{\nu}{4} \|\D{\ecko_h^m}\|_0^2 + \fz{\alpha_1^2}{\nu}\|\rcko_{hi}^m\|_{-1}^2 + \alpha_1\sum_{n=1}^{m-1} \| \rcko_{hi}^{n+1} - \rcko_{hi}^n\|_0 \|\D{\ecko_h^n}\|_0 + \fz{1}{2}\|\rcko_{hi}^1\|_{-1}^2 + \fz{1}{2}\|\ecko_h^0\|_1^2 \notag\\
& \quad \le \fz{\nu}{4} \|\D{\ecko_h^m}\|_0^2 + \alpha_1\sum_{n=1}^{m-1} \| \rcko_{hi}^{n+1} - \rcko_{hi}^n\|_0 \|\D{\ecko_h^n}\|_0 + c_{\nu,\varepsilon} c_{w,s} \bigl( \Delta t^2 + h^2 \bigr) \qquad \mbox{(by~\eqref{ieq:r3},\eqref{ieq:r4},\eqref{ieq:c_I}, Thm.\ref{thm:error_estimates}).}
\label{ieq:r34_pressure_prep}
\end{align}
Applying H\" older's inequality, we have
\begin{align*}
& \bigl\| \tr\Ccko^{n+1} (\Ccko^{n+1}-\Ccko^n) - \tr\Ccko^n (\Ccko^n-\Ccko^{n-1}) \bigr\|_0
= \biggl\| \int_{t^n}^{t^{n+1}} \prz{}{t} \bigl\{ \tr\Ccko(t) \bigl[ \Ccko (t) -\Ccko (t-\Delta t) \bigr] \bigr\} dt \biggr\|_0 \\
& \quad \le \biggl\| \int_{t^n}^{t^{n+1}} \tr\prz{\Ccko}{t}(t) \bigl[ \Ccko (t) -\Ccko (t-\Delta t) \bigr] \, dt \biggr\|_0 + \biggl\| \int_{t^n}^{t^{n+1}} \tr\Ccko (t) \biggl[ \prz{\Ccko}{t} (t) - \prz{\Ccko}{t} (t-\Delta t) \biggr] \, dt \biggr\|_0 \\
& \quad = \biggl\| \int_{t^n}^{t^{n+1}} \tr\prz{\Ccko}{t}(t) \, dt \int_{t-\Delta t}^t \prz{\Ccko}{t} (s) \, ds \biggr\|_0 + \biggl\| \int_{t^n}^{t^{n+1}} \tr\Ccko (t) \, dt \int_{t-\Delta t}^t \prz{^2\Ccko}{t^2} (s) \, ds \biggr\|_0 \\
& \quad \le \int_{t^n}^{t^{n+1}} \biggl\| \tr\prz{\Ccko}{t}(t) \biggr\|_{0,4} dt \int_{t^{n-1}}^{t^{n+1}} \biggl\| \prz{\Ccko}{t} (s) \biggr\|_{0,4} ds + \int_{t^n}^{t^{n+1}} \|\tr\Ccko (t)\|_{0,\infty} \, dt \int_{t^{n-1}}^{t^{n+1}} \biggl\| \prz{^2\Ccko}{t^2} (s) \biggr\|_0 ds \\
& \quad \le \Delta t^{3/4} \biggl( \int_{t^n}^{t^{n+1}} \biggl\| \prz{\Ccko}{t}(t) \biggr\|_{0,4}^4 dt \biggr)^{1/4} (2\Delta t)^{3/4} \biggl( \int_{t^{n-1}}^{t^{n+1}} \biggl\| \prz{\Ccko}{t}(s) \biggr\|_{0,4}^4 ds \biggr)^{1/4} + d \Delta t \|\Ccko \|_{C(L^\infty)} \sqrt{2\Delta t} \|\Ccko\|_{H^2 (t^{n-1},t^{n+1}; L^2)} \\
& \quad \le c_s \Delta t^{3/2} \biggl( \biggl\| \prz{\Ccko}{t} \biggr\|_{L^4(t^{n-1},t^{n+1}; L^4)}^2 + \|\Ccko\|_{H^2 (t^{n-1},t^{n+1}; L^2)} \biggr), \\
& \| \tr\Ccko^{n+1}\Xicko^n - \tr\Ccko^n\Xicko^{n-1} \|_0 
\le \| (\tr\Ccko^{n+1} - \tr\Ccko^n) \Xicko^n \|_0 + \| \tr\Ccko^n (\Xicko^n - \Xicko^{n-1}) \|_0 \\
& \quad \le \| \tr\Ccko^{n+1} - \tr\Ccko^n \|_{0,3} \|\Xicko^n \|_{0,6} + \| \tr\Ccko^n \|_{0,3} \| \Xicko^n - \Xicko^{n-1} \|_{0,6} \\
& \quad \le \sqrt{\Delta t} \, \biggl\| \prz{(\tr\Ccko)}{t} \biggr\|_{L^2(t^n,t^{n+1}; L^3)} \|\Xicko^n \|_{0,6} + \| \Ccko^n \|_1 \sqrt{\Delta t} \, \biggl\| \prz{\Xicko}{t} \biggr\|_{L^2(t^{n-1},t^n; L^6)} \\
& \quad \le c \sqrt{\Delta t} \bigl( \| \Ccko \|_{H^1(t^n,t^{n+1}; H^1)} \|\Xicko^n \|_1 + \| \Ccko^n \|_1 \| \Xicko \|_{H^1(t^{n-1},t^n; H^1)} \bigr) \\
& \quad \le c \sqrt{\Delta t} ( \| \Ccko \|_{H^1(t^n,t^{n+1}; H^1)} \alpha_{32} h \|\Ccko^n \|_2 + \| \Ccko^n \|_1 \alpha_{32} h \| \Ccko \|_{H^1(t^{n-1},t^n; H^2)} \bigr) \\
& \quad \le c_s h \sqrt{\Delta t} \| \Ccko \|_{H^1(t^{n-1},t^{n+1}; H^2)}, \\
& \| \tr\Ccko^{n+1}\Ecko_h^n - \tr\Ccko^n\Ecko_h^{n-1} \|_0 
\le \| (\tr\Ccko^{n+1} - \tr\Ccko^n) \Ecko_h^n \|_0 + \| \tr\Ccko^n (\Ecko_h^n - \Ecko_h^{n-1}) \|_0 \\
& \quad \le \| \tr\Ccko^{n+1} - \tr\Ccko^n \|_{0,3} \| \Ecko_h^n \|_{0,6} + \| \tr\Ccko^n \|_{0,\infty} \| \Ecko_h^n - \Ecko_h^{n-1} \|_0 \\
& \quad \le \sqrt{\Delta t} \, \biggl\| \prz{(\tr\Ccko)}{t} \biggr\|_{L^2(t^n,t^{n+1};L^3)} \| \Ecko_h^n \|_1 + \| \tr\Ccko^n \|_{0,\infty} \Delta t \, \| \ol{D}_{\Delta t} \Ecko_h^n \|_0 \\
& \quad \le c_{\nu,\varepsilon} c_{w,s} \sqrt{\Delta t} \bigl[ (\Delta t + h) \| \Ccko \|_{H^1(t^n,t^{n+1};H^1)} + \sqrt{\Delta t} \, \| \ol{D}_{\Delta t} \Ecko_h^n \|_0 \bigr] \qquad \mbox{(by Thm.\ref{thm:error_estimates}).}
\end{align*}
Hence, $\| \rcko_{h3}^{n+1} - \rcko_{h3}^n\|_0$ is evaluated as follows.
\begin{align}
& \| \rcko_{h3}^{n+1} - \rcko_{h3}^n\|_0
 = \| (\tr \Ccko^{n+1})(\Ccko^{n+1}-\Ccko^n + \Xicko^n-\Ecko_h^n) - (\tr \Ccko^n)(\Ccko^n-\Ccko^{n-1} + \Xicko^{n-1}-\Ecko_h^{n-1}) \|_0 \notag\\
& \le \| (\tr \Ccko^{n+1})(\Ccko^{n+1}-\Ccko^n) - (\tr \Ccko^n)(\Ccko^n-\Ccko^{n-1}) \|_0 + \| (\tr \Ccko^{n+1}) \Xicko^n - (\tr \Ccko^n) \Xicko^{n-1} \|_0 \notag\\
& \quad + \| (\tr \Ccko^{n+1}) \Ecko_h^n - (\tr \Ccko^n) \Ecko_h^{n-1} \|_0 \notag\\
& \le c_{\nu,\varepsilon} c_{w,s} \sqrt{\Delta t} \biggl[ \Delta t \biggl\| \prz{\Ccko}{t} \biggr\|_{L^4(t^{n-1},t^{n+1}; L^4)}^2 
+ \Delta t \|\Ccko\|_{H^2 (t^{n-1},t^{n+1}; L^2)} 
+ (\Delta t + h) \| \Ccko \|_{H^1(t^{n-1},t^{n+1}; H^2)}
+ \sqrt{\Delta t} \, \| \ol{D}_{\Delta t} \Ecko_h^n \|_0 \biggr].
\label{ieq:r3_pressure_prep}
\end{align}
Combining~\eqref{ieq:r3_pressure_prep} with~\eqref{ieq:r34_pressure_prep} with~$i=3$, we get
\begin{align}
& \Delta t \sum_{n=1}^m \lA \rcko_{h3}^n, \ol{D}_{\Delta t}\ecko_h^n \rA \notag\\
& \le \fz{\nu}{4} \|\D{\ecko_h^m}\|_0^2 + c_{\nu,\varepsilon} c_{w,s} \biggl\{ \bigl( \Delta t^2 + h^2 \bigr) + \sum_{n=1}^{m-1} \biggl[ \Delta t \biggl\| \prz{\Ccko}{t} \biggr\|_{L^4(t^{n-1},t^{n+1}; L^4)}^2 + \Delta t \|\Ccko\|_{H^2 (t^{n-1},t^{n+1}; L^2)}
\notag\\
& \qquad\qquad\qquad\qquad\qquad\qquad\qquad\qquad\qquad\quad
+ (\Delta t + h) \| \Ccko \|_{H^1(t^{n-1},t^{n+1}; H^2)} + \sqrt{\Delta t} \, \| \ol{D}_{\Delta t} \Ecko_h^n \|_0 \biggr] \sqrt{\Delta t} \|\D{\ecko_h^n}\|_0 \biggr\} \notag\\
& \le \fz{\nu}{4} \|\D{\ecko_h^m}\|_0^2 + c_{\nu,\varepsilon} c_{w,s} \biggl\{ ( \Delta t^2 + h^2 ) + \|\ecko_h\|_{\ell^2(H^1)}^2 + \| \ol{D}_{\Delta t} \Ecko_h \|_{\ell^2(L^2)}^2 + 2 ( \Delta t^2 + h^2 ) \| \Ccko \|_{H^1(H^2)}^2 \notag\\
& \qquad\qquad\qquad\qquad\qquad\qquad\qquad\qquad\qquad\qquad\qquad\qquad\qquad\qquad\qquad\qquad
 + \Delta t^2 \biggl( \biggl\| \prz{\Ccko}{t} \biggr\|_{L^4(L^4)}^4 + \|\Ccko\|_{H^2 (L^2)}^2 \biggr) \biggr\} \notag\\
& \le \fz{\nu}{4} \|\D{\ecko_h^m}\|_0^2 + c_{\nu,\varepsilon}^\prime c_{w,s}^\prime (\Delta t^2 + h^2 ),
\label{ieq:r3_pressure_final}
\end{align}
where in the last inequality we have employed Theorem~\ref{thm:error_estimates} and the relation $[L^2(0,T; H^1(\Omega))\cap H^1(0,T; L^2(\Omega))] \hookrightarrow L^4(0,T;L^4(\Omega))$ yielding the inequality $\| \pz\Ccko/\pz t \|_{L^4(L^4)} \le c \| \pz\Ccko/\pz t \|_{L^2(H^1)\cap H^1(L^2)} \le c \| \Ccko \|_{H^1(H^1)\cap H^2(L^2)} \le c_s$.
Thus, the first inequality of~\eqref{ieq:r3_r4_pressure} is proved.
We prove the other inequality of~\eqref{ieq:r3_r4_pressure}.
For $\| \rcko_{h4}^{n+1} - \rcko_{h4}^n\|_0$ we have
\begin{align}
& \| \rcko_{h4}^{n+1} - \rcko_{h4}^n\|_0
= \| [\tr (\Xicko^{n+1}-\Ecko_h^{n+1})] \Ccko_h^n - [\tr (\Xicko^n-\Ecko_h^n)] \Ccko_h^n + [\tr (\Xicko^n-\Ecko_h^n)] \Ccko_h^n - [\tr (\Xicko^n-\Ecko_h^n)] \Ccko_h^{n-1} \|_0 \notag\\
& \quad = \Bigl\| [\tr (\Xicko^{n+1}- \Xicko^n)] \Ccko_h^n - [\tr (\Ecko_h^{n+1}-\Ecko_h^n)] \Ccko_h^n + [\tr (\Xicko^n-\Ecko_h^n)]\Bigl[ \Delta t \ol{D}_{\Delta t}\Ecko_h^n - (\Xicko^n - \Xicko^{n-1}) + (\Ccko^n - \Ccko^{n-1}) \Bigr] \Bigr\|_0 \notag\\
& \quad \le c \Bigl[ \|\Ccko_h^n\|_{0,\infty} (\| \Xicko^{n+1}- \Xicko^n \|_0 + \Delta t \| \ol{D}_{\Delta t} \Ecko_h^{n+1} \|_0) + \| \Xicko^n-\Ecko_h^n \|_{0,\infty} ( \Delta t \| \ol{D}_{\Delta t}\Ecko_h^n \|_0 + \| \Xicko^n - \Xicko^{n-1} \|_0) \notag\\
& \quad \qquad + (\| \Xicko^n \|_0 + \| \Ecko_h^n \|_0) \| \Ccko^n - \Ccko^{n-1} \|_{0,\infty} \Bigr] \notag\\
& \quad \le c \Bigl[ (2\|\Ccko\|_{C(L^\infty)}+1) \Bigl\{ \sqrt{\Delta t} \| \Xicko \|_{H^1(t^n,t^{n+1};L^2)} + \Delta t (\| \ol{D}_{\Delta t}\Ecko_h^{n+1} \|_0 + \| \ol{D}_{\Delta t}\Ecko_h^n \|_0 ) + \sqrt{\Delta t} \| \Xicko \|_{H^1(t^{n-1},t^n;L^2)} \Bigr\} \notag\\
& \quad\qquad + (\| \Xicko^n \|_0 + \| \Ecko_h^n \|_0) \sqrt{\Delta t} \| \Ccko \|_{H^1(t^{n-1},t^n;L^\infty)} \Bigr] \notag\\
& \quad \le c_{\nu,\varepsilon} c_{w,s} \sqrt{\Delta t} \bigl[ \alpha_{32} h \| \Ccko \|_{H^1(t^{n-1},t^{n+1};H^2)} + \sqrt{\Delta t} (\| \ol{D}_{\Delta t}\Ecko_h^{n+1} \|_0 + \| \ol{D}_{\Delta t}\Ecko_h^n \|_0 ) \notag\\
& \quad\qquad 
+ ( \alpha_{32} h \| \Ccko^n \|_2 + \| \Ecko_h^n \|_0 ) \| \Ccko \|_{H^1(t^{n-1},t^n;H^2)} \bigr] \notag\\
& \quad \le c_{\nu,\varepsilon}^\prime c_{w,s}^\prime \sqrt{\Delta t} \Bigl[ \sqrt{\Delta t} (\| \ol{D}_{\Delta t}\Ecko_h^{n+1} \|_0 + \| \ol{D}_{\Delta t}\Ecko_h^n \|_0) + (\Delta t + h) \| \Ccko \|_{H^1(t^{n-1},t^{n+1};H^2)} \Bigr],
\label{ieq:r4_pressure_prep}
\end{align}
where we have used the estimates,
\begin{align*}
& \Ccko_h^n-\Ccko_h^{n-1} = (\Ecko_h^n-\Xicko^n+\Ccko^n) - (\Ecko_h^{n-1}-\Xicko^{n-1}+\Ccko^{n-1}) =  \Delta t \ol{D}_{\Delta t}\Ecko_h^n - (\Xicko^n - \Xicko^{n-1}) + (\Ccko^n - \Ccko^{n-1}), \\
& \|\Xicko^n-\Ecko_h^n\|_{0,\infty} = \|\Ccko^n-\Ccko_h^n\|_{0,\infty} \le \|\Ccko^n\|_{0,\infty} + \|\Ccko_h^n\|_{0,\infty} \le 2 \|\Ccko\|_{C(L^\infty)} + 1 \qquad \mbox{(by Thm.\ref{thm:error_estimates}).}
\end{align*}
Combining~\eqref{ieq:r4_pressure_prep} with~\eqref{ieq:r34_pressure_prep} with $i=4$, we have
\begin{align}
& \Delta t \sum_{n=1}^m \lA \rcko_{h4}^n, \ol{D}_{\Delta t}\ecko_h^n \rA \notag\\
& \quad \le \fz{\nu}{4} \|\D{\ecko_h^m}\|_0^2 + c_{\nu,\varepsilon} c_{w,s} \biggl\{ \bigl( \Delta t^2 + h^2 \bigr) \notag\\
& \quad\qquad
+ \sum_{n=1}^{m-1}
\Bigl[ \sqrt{\Delta t} (\| \ol{D}_{\Delta t}\Ecko_h^{n+1} \|_0 + \| \ol{D}_{\Delta t}\Ecko_h^n \|_0) + (\Delta t + h) \| \Ccko \|_{H^1(t^{n-1},t^{n+1};H^2)} \Bigr] \sqrt{\Delta t} \|\D{\ecko_h^n}\|_0 \biggr\} \notag\\
& \quad \le \fz{\nu}{4} \|\D{\ecko_h^m}\|_0^2  
+ c_{\nu,\varepsilon} c_{w,s} \Bigl\{ ( \Delta t^2 + h^2 ) + \|\ecko_h\|_{\ell^2(H^1)}^2 
+ ( \Delta t^2 + h^2 ) \| \Ccko \|_{H^1(H^2)}^2 + \| \ol{D}_{\Delta t} \Ecko_h \|_{\ell^2(L^2)}^2 \Bigr\} \notag\\
& \quad \le \fz{\nu}{4} \|\D{\ecko_h^m}\|_0^2 + c_{\nu,\varepsilon}^\prime c_{w,s}^\prime (\Delta t^2 + h^2 ) \qquad \mbox{(by Thm.\ref{thm:error_estimates}),}
\label{ieq:r4_pressure_final}
\end{align}
which is the other inequality of~\eqref{ieq:r3_r4_pressure}.
\end{proof}
%
%
%
%
\subsection{Proof of Theorem~\ref{thm:error_estimates_pressure}}
Let $p_h^0 \defeq [\Pi_h^{\rm SP} (\ucko^0, 0, \Ccko^0)]_2$, which leads to $(\ucko_h^0, p_h^0, \Ccko_h^0) = [\Pi_h^{\rm SP} (\ucko^0, 0, \Ccko^0)]$.
Substituting $(\ol{D}_{\Delta t}\ecko_h^n, 0) \in V_h\times Q_h$ into $(\vcko_h, q_h)$ in~\eqref{eq:error_up} and using
\[
\fz{\ecko_h^n-\ecko_h^{n-1}\circ X_1^n}{\Delta t} = \ol{D}_{\Delta t}\ecko_h^n + \fz{\ecko_h^{n-1}-\ecko_h^{n-1}\circ X_1^n}{\Delta t},
\]
we have
\begin{align}
\| \ol{D}_{\Delta t} \ecko_h^n \|_0^2 + \nu a_u\bigl( e_h^n, \ol{D}_{\Delta t}\ecko_h^n \bigr) + b(\ol{D}_{\Delta t}\ecko_h^n, \epsilon_h^n) = \lA \rcko_h^n, \ol{D}_{\Delta t}\ecko_h^n \rA - \fz{1}{\Delta t} ( \ecko_h^{n-1}-\ecko_h^{n-1}\circ X_1^n, \ol{D}_{\Delta t}\ecko_h^n ).
\label{eq:proof_error_estimates_pressure_main_prep0}
\end{align}
On the other hand, setting $\vcko_h = \mathbf{0} \in V_h$ in~\eqref{eq:error_up}, we have for $n=1,\ldots, N_T$
\begin{align}
b(\ecko_h^n, q_h) - \mathcal{S}_h(\epsilon_h^n, q_h) = 0, \quad \forall q_h \in Q_h,
\label{eq:proof_error_estimates_pressure_b_S_prep}
\end{align}
From the definitions of $(\ucko_h^0, p_h^0, \Ccko_h^0)$ and $(\hat{\ucko}_h^0,\hat{p}_h^0,\hat{\Ccko}_h^0)$ we have
\begin{align*}
b(\ecko_h^0,q_h) - \mathcal{S}_h(\epsilon_h^0, q_h)
= b(\ucko_h^0, q_h) - \mathcal{S}_h(p_h^0, q_h) - \bigl\{ b(\hat{\ucko}_h^0, q_h) - \mathcal{S}_h(\hat{p}_h^0, q_h) \bigr\} 
= b(\ucko^0, q_h) - b(\ucko^0, q_h) =0,\quad \forall q_h \in Q_h,
\end{align*}
which implies that \eqref{eq:proof_error_estimates_pressure_b_S_prep} holds also for $n=0$.
Hence, we get for $n=1,\ldots, N_T$
\begin{align*}
b(\ol{D}_{\Delta t}\ecko_h^n, q_h) - \mathcal{S}_h(\ol{D}_{\Delta t}\epsilon_h^n, q_h) = 0, \quad \forall q_h \in Q_h,
\end{align*}
which yields
\begin{align}
b(\ol{D}_{\Delta t}\ecko_h^n, \epsilon_h^n) - \mathcal{S}_h(\ol{D}_{\Delta t}\epsilon_h^n, \epsilon_h^n) = 0
\label{eq:proof_error_estimates_pressure_b_S_final}
\end{align}
by setting $q_h = \epsilon_h^n \in Q_h$.
Subtracting~\eqref{eq:proof_error_estimates_pressure_b_S_final} from~\eqref{eq:proof_error_estimates_pressure_main_prep0}, we have for $n=1,\ldots, N_T$
\begin{align}
\| \ol{D}_{\Delta t} \ecko_h^n \|_0^2 + \nu a_u\bigl( e_h^n, \ol{D}_{\Delta t}\ecko_h^n \bigr) + \mathcal{S}_h(\ol{D}_{\Delta t}\epsilon_h^n, \epsilon_h^n) = \lA \rcko_h^n, \ol{D}_{\Delta t}\ecko_h^n \rA - \biggl( \fz{\ecko_h^{n-1}-\ecko_h^{n-1}\circ X_1^n}{\Delta t}, \ol{D}_{\Delta t}\ecko_h^n \biggr).
\label{eq:proof_error_estimates_pressure_main_prep1}
\end{align}
From the estimates,
\begin{align*}
\nu a_u(\ecko_h^n, \ol{D}_{\Delta t}\ecko_h^n) 
& = \ol{D}_{\Delta t} \Bigl( \fz{\nu}{2} a_u(\ecko_h^n, \ecko_h^n) \Bigr) + \fz{\nu\Delta t}{2} a_u( \ol{D}_{\Delta t}\ecko_h^n, \ol{D}_{\Delta t}\ecko_h^n ) 
\ge \ol{D}_{\Delta t} \bigl( \nu \|\D{\ecko_h^n}\|_0^2 \bigr), \\
\mathcal{S}_h ( \ol{D}_{\Delta t}\epsilon_h^n, \epsilon_h^n) & = \ol{D}_{\Delta t} \Bigl( \fz{1}{2} \mathcal{S}_h( \epsilon_h^n, \epsilon_h^n) \Bigr) +\fz{\Delta t}{2} \mathcal{S}_h( \ol{D}_{\Delta t} \epsilon_h^n, \ol{D}_{\Delta t}\epsilon_h^n )
\ge \ol{D}_{\Delta t} \Bigl( \fz{\delta_0}{2} |\epsilon_h^n|_h^2 \Bigr), \\
\Bigl| \fz{1}{\Delta t} ( \ecko_h^{n-1}-\ecko_h^{n-1}\circ X_1^n, \ol{D}_{\Delta t}\ecko_h^n ) \Bigr|
& \le \fz{1}{\Delta t} \| \ecko_h^{n-1} - \ecko_h^{n-1}\circ X_1^n\|_0 \|\ol{D}_{\Delta t}\ecko_h^n\|_0 
\le \alpha_{41} \| \wcko^n \|_{0,\infty} \|\ecko_h^{n-1}\|_1 \|\ol{D}_{\Delta t}\ecko_h^n\|_0 \notag\\
& \le \alpha_{41} \| \wcko^n \|_{0,\infty} \alpha_1 \|\D{\ecko_h^{n-1}}\|_0 \|\ol{D}_{\Delta t}\ecko_h^n\|_0 
\le c_w \|\D{\ecko_h^{n-1}}\|_0^2 + \fz{1}{6}\|\ol{D}_{\Delta t}\ecko_h^n\|_0^2, 
\end{align*}
the equality~\eqref{eq:proof_error_estimates_pressure_main_prep1} leads to, for $n=1,\ldots, N_T$,
\begin{align}
\ol{D}_{\Delta t} \Bigl( \nu \|\D{\ecko_h^n}\|_0^2 + \fz{\delta_0}{2} |\epsilon_h^n|_h^2 \Bigr)
+ \fz{5}{6} \| \ol{D}_{\Delta t} \ecko_h^n \|_0^2
\le \lA \rcko_h^n, \ol{D}_{\Delta t}\ecko_h^n \rA + c_w \|\D{\ecko_h^{n-1}}\|_0^2.
\label{eq:proof_error_estimates_pressure_main_prep2}
\end{align}
Let $m~(1\le m \le N_T)$ be any integer.
Summing up~\eqref{eq:proof_error_estimates_pressure_main_prep2} for $n=1,\ldots, m$ and using Lemma~\ref{lem:rh_pressure}, we have
\begin{align}
\fz{\nu}{2} \|\D{\ecko_h^m}\|_0^2 + \fz{\delta_0}{2} |\epsilon_h^m|_h^2
+ \fz{\Delta t}{2} \sum_{n=1}^m \| \ol{D}_{\Delta t} \ecko_h^n \|_0^2
\le \fz{c_w}{\nu} \Delta t \sum_{n=0}^{m-1} \nu \|\D{\ecko_h^n}\|_0^2 + c_{\nu,\varepsilon} c_{w,s} (\Delta t^2 + h^2).
\label{eq:proof_error_estimates_pressure_main_prep3}
\end{align}
From Lemma~\ref{lem:Gronwall} with
\begin{align*}
x_n & = \fz{\nu}{2} \|\D{\ecko_h^n}\|_0^2 + \fz{\delta_0}{2} |\epsilon_h^n|_h^2,
&
y_n & = \fz{1}{2} \| \ol{D}_{\Delta t} \ecko_h^n \|_0^2,
&
\alpha & = \fz{2c_w}{\nu},
&
\beta & = c_{\nu,\varepsilon} c_{w,s} (\Delta t^2 + h^2),
\end{align*}
we have
\begin{align}
\| \ol{D}_{\Delta t} \ecko_h \|_{\ell^2(L^2)} \le c_{\nu,\varepsilon}^\prime c_{w,s}^\prime (\Delta t + h).
\label{ieq:D_eh_L2}
\end{align}
The first inequality of~\eqref{ieq:main_results_pressure} is obtained by combining the inequality above with the estimate
\begin{align*}
\Bigl\| \ol{D}_{\Delta t}\ucko_h^n - \prz{\ucko^n}{t} \Bigr\|_0 
& \le \|\ol{D}_{\Delta t}\ecko_h^n\|_0 + \|\ol{D}_{\Delta t}\etacko^n\|_0 + \Bigl\| \ol{D}_{\Delta t}\ucko^n - \prz{\ucko^n}{t} \Bigr\|_0 \\
& \le \|\ol{D}_{\Delta t}\ecko_h^n\|_0 + \fz{\alpha_{31} h}{\nu\sqrt{\Delta t}}\|(\ucko,p)\|_{H^1(t^{n-1},t^n; H^2\times H^1)} + \sqrt{\fz{\Delta t}{3}} \Bigl\| \prz{^2\ucko}{t^2} \Bigr\|_{L^2(t^{n-1},t^n; L^2)}.
\end{align*}
\par
The other inequality of~\eqref{ieq:main_results_pressure} is proved as follows.
We have
\begin{align*}
\|\epsilon_h^n\|_0 & \le \|(\ecko_h^n,\epsilon_h^n)\|_{V\times Q} \le \fz{1}{\nu\alpha_{30}} \sup_{(\vcko_h,q_h)\in V_h\times Q_h} \fz{\mathcal{A}_h((\ucko_h^n,\epsilon_h^n), (\vcko_h,q_h))}{\|(\vcko_h,q_h)\|_{V\times Q}} \\
& = \fz{1}{\nu\alpha_{30}} \sup_{(\vcko_h,q_h)\in V_h\times Q_h} \fz{\lA \rcko_h^n, \vcko_h \rA - \fz{1}{\Delta t}(\ecko_h^n - \ecko_h^{n-1} \circ X_1^n, \vcko_h)}{\|(\vcko_h,q_h)\|_{V\times Q}} \\
& \le \fz{1}{\nu\alpha_{30}} \Bigl[ \|\rcko_{h1}^n\|_0 + \|\rcko_{h2}^n\|_0 + \|\rcko_{h3}^n\|_{-1} + \|\rcko_{h4}^n\|_{-1} + \|\ol{D}_{\Delta t}\ecko_h^n\|_0 + \fz{1}{\Delta t} \| \ecko_h^{n-1}-\ecko_h^{n-1}\circ X_1^n \|_0 \Bigr] \\
& \le \fz{c_s}{\nu\alpha_{30}} \Bigl[ \sqrt{\Delta t} \bigl( \|\ucko\|_{Z^2(t^{n-1},t^n)} + \|\Ccko\|_{H^1(t^{n-1},t^n; L^2)} \bigr) + \fz{h}{\nu \sqrt{\Delta t}} \| (\ucko, p) \|_{H^1(t^{n-1},t^n; H^2\times H^1)} \\
& \qquad\qquad + \|\ol{D}_{\Delta t}\ecko_h^n\|_0 + \| \ecko_h^{n-1} \|_1 + \|\Ecko_h^n\|_0 + \|\Ecko_h^{n-1}\|_0 + h \Bigr] \quad \mbox{(by~\eqref{ieq:r1}--\eqref{ieq:r4})},
\end{align*}
which implies the second inequality of~\eqref{ieq:main_results_pressure} from Theorem~\ref{thm:error_estimates}, \eqref{ieq:D_eh_L2} and the estimate
\begin{align*}
\phantom{MMMMa}
\|p_h-p\|_{\ell^2(L^2)} \le \|\epsilon_h\|_{\ell^2(L^2)} + \|\hat{p}_h-p\|_{\ell^2(L^2)} 
\le \|\epsilon_h\|_{\ell^2(L^2)} + \sqrt{T} \, \fz{\alpha_{31}}{\nu} h \|(\ucko,p)\|_{C(H^2\times H^1)}.
\phantom{MMMa}
\qed
\end{align*}
%
%
%
%
%
\section{Numerical experiments}\label{sec:numerics}
In this section we present numerical results by scheme~\eqref{lin_scheme} in order to confirm the theoretical convergence order.
We refer to \cite{Miz-2015} for the detailed description of the algorithm that has been used to perform the numerical simulations.
Further numerical experiments for linear scheme~\eqref{lin_scheme} as well as for the nonlinear scheme that has been discussed in our previous paper~\cite{LMNT-Peterlin_Oseen_Part_I}, Part I, can also be found in \cite{Miz-2015}.
\begin{example}
In problem~\eqref{model} we set $\Omega=(0, 1)^2$ and $T=0.5$, and we consider three cases for the pair of~$\nu$ and~$\varepsilon$.
Firstly we take both viscosities to be equal $10^{-1}$, i.e., $(\nu,\varepsilon)=(10^{-1}, 10^{-1})$.
Secondly, we consider the case $(\nu,\varepsilon)=(10^{-1}, 10^{-3})$, since the elastic stress viscosity is typically much smaller than the fluid viscosity.
Lastly, we set $(\nu,\varepsilon)=(1, 0)$.
Although the non-diffusive case $\varepsilon=0$ is out of the scope of theoretical analysis of this paper, we dare to carry out the computation to see the performance of scheme~\eqref{lin_scheme}.
The functions $\fko$, $\Fko$, $\ucko^0$ and $\Ccko^0$ are given such that the exact solution to \eqref{model} is as follows:
\begin{equation}\label{exact_solution}
\begin{aligned}
\mathbf{u} (x,t) &= \( \pd{\psi}{x_2} (x,t), -\pd{\psi}{x_1} (x,t) \),\quad p(x,t) = \sin \{ \pi (x_1 + 2 x_2 + t) \},\\
C_{11}(x,t) &=\frac{1}{2} \sin^2 (\pi x_1) \sin^2 (\pi x_2) \sin \{\pi(x_1+t)\} + 1,\\
C_{22}(x,t) &=\frac{1}{2} \sin^2 (\pi x_1) \sin^2 (\pi x_2) \sin \{\pi(x_2+t)\} + 1,\\
C_{12}(x,t) &=\frac{1}{2} \sin^2 (\pi x_1) \sin^2 (\pi x_2) \sin \{\pi(x_1+x_2+t)\} \ (=C_{21}(x,t)), \\
\psi(x,t) & \defeq \frac{\sqrt{3}}{2\pi} \sin^2 (\pi x_1) \sin^2 (\pi x_2) \sin \{ \pi (x_1+x_2 + t) \}.
\end{aligned}
\end{equation}
\end{example}
\par
Proposition~\ref{prop:existence_uniqueness} and Theorems~\ref{thm:error_estimates} and~\ref{thm:error_estimates_pressure} hold for any fixed positive constant~$\delta_0$.
Here we simply fix $\delta_0=1$.
Let $N$ be the division number of each side of the square domain.
We set $N=16, 32, 64, 128$ and $256$, and (re)define $h \defeq 1/N$.
The time increment is set as $\Delta t = h/2$.
To solve Example we employ scheme~\eqref{lin_scheme} with~$(\ucko_h^0, \Ccko_h^0) = [\Pi_h^{\rm SP}(\ucko^0, 0, \Ccko^0)]_{1,3}$.
\par
For the solution~$(\ucko_h, p_h, \Ccko_h)$ of scheme~\eqref{lin_scheme} and the exact solution~$(\ucko, p, \Ccko)$ given by \eqref{exact_solution} we define the relative errors~$Er\,i$, $i=1,\ldots, 6$, by
\begin{align*}
Er\,1 & =\frac{\|\ucko_h-\Pi_h\ucko\|_{\ell^\infty(L^2)}}{\|\Pi_h\ucko\|_{\ell^\infty(L^2)}}, &
Er\,2 & =\frac{\|\ucko_h-\Pi_h\ucko\|_{\ell^2(H^1)}}{\|\Pi_h\ucko\|_{\ell^2(H^1)}}, &
Er\,3 & =\frac{\|p_h-\Pi_hp\|_{\ell^2(L^2)}}{\|\Pi_hp\|_{\ell^2(L^2)}}, \\
Er\,4 & =\frac{|p_h-\Pi_hp|_{\ell^2(|\cdot|_h)}}{\|\Pi_hp\|_{\ell^2(L^2)}}, &
Er\,5 & =\frac{\|\Ccko_h-\Pi_h\Ccko\|_{\ell^\infty(L^2)}}{\|\Pi_h\Ccko\|_{\ell^\infty(L^2)}}, &
Er\,6 & =\frac{\|\Ccko_h-\Pi_h\Ccko\|_{\ell^2(H^1)}}{\|\Pi_h\Ccko\|_{\ell^2(H^1)}},
\end{align*}
where the same symbol~$\Pi_h$ has been employed as the scalar and vector versions of the {Lagrange} interpolation operator.
\par
The values of the errors and the slopes are presented in the tables below, while the corresponding figures show the graphs of the errors versus $h$ in logarithmic scale.
Table~\ref{table:symbols} summarizes the symbols used in the figures.
Tables~\&~Figures~\ref{table:first_case}, \ref{table:second_case} and~\ref{table:third_case} present the results for the cases~$(\nu, \varepsilon)=(10^{-1}, 10^{-1})$, $(10^{-1}, 10^{-3})$ and~$(1, 0)$, respectively.
\par
For all the cases it is confirmed that all the errors except~$Er\,6$ for~$(\nu,\varepsilon) = (1, 0)$ are almost of the first order in~$h$.
These results support Theorems~\ref{thm:error_estimates} and~\ref{thm:error_estimates_pressure}.
Since there is no diffusion for~$\Ccko$ in equation~\eqref{model_Ccko} in the case~$(\nu,\varepsilon) = (1, 0)$, it is natural that the slope of~$Er\,6$ does not attain~$1$.
While the theorems are not proved for~$\varepsilon = 0$, scheme~\eqref{lin_scheme} has worked well in the numerical experiments.
\begin{table}[!htbp]
\centering
\caption{Symbols used in the figures.}
\label{table:symbols}
\begin{tabular}{cccccccc}
\toprule
\multicolumn{2}{c}{$\ucko_h$} && \multicolumn{2}{c}{$p_h$} && \multicolumn{2}{c}{$\Ccko_h$} \\ \cmidrule{1-2} \cmidrule{4-5} \cmidrule{7-8}
{\LARGE $\circ$} & {\Large $\bullet$} && $\triangle$ & {\large $\blacktriangle$} && $\Box$ & $\blacksquare$ \\
$Er\,1$ & $Er\,2$ && $Er\,3$ & $Er\,4$ && $Er\,5$ & $Er\,6$ \\
\bottomrule
\end{tabular}
\end{table}
\begin{figure}[!htbp]
\centering
\begin{tabular}{rrrrr}
\toprule
 $h$  &  $Er\,1$ &  slope  & $Er\,2$ & slope  \\  \midrule
 $1/16$   & $6.29 \times 10^{-2}$ & --     & $7.94 \times 10^{-2}$ & --         \\
 $1/32$   & $2.21 \times 10^{-2}$ & $1.51$ & $3.14 \times 10^{-2}$ & $1.34$  \\
 $1/64$   & $8.98 \times 10^{-3}$ & $1.30$ & $1.32 \times 10^{-2}$ & $1.25$  \\
 $1/128$ & $4.07 \times 10^{-3}$ & $1.14$ & $6.35 \times 10^{-3}$ & $1.05$  \\
 $1/256$ & $1.95 \times 10^{-3}$ & $1.07$ & $2.86 \times 10^{-3}$ & $1.15$  \\
\cmidrule{1-5}
 $h$  &  $Er\,3$ &  slope  & $Er\,4$ & slope \\
\cmidrule{1-5}
 $1/16$   & $2.02 \times 10^{-1}$ & --         & $1.70 \times 10^{-1}$ & -- \\
 $1/32$   & $7.11 \times 10^{-2}$ & $1.50$ & $4.99 \times 10^{-2}$ & $1.77$ \\
 $1/64$   & $2.67 \times 10^{-2}$ & $1.41$ & $1.86 \times 10^{-2}$ & $1.42$ \\
 $1/128$ & $1.11 \times 10^{-2}$ & $1.27$ & $8.39 \times 10^{-3}$ & $1.15$\\
 $1/256$ & $5.01 \times 10^{-3}$ & $1.15$ & $3.69 \times 10^{-3}$ & $1.19$\\
\cmidrule{1-5}
 $h$       &  $Er\,5$ &  slope  &  $Er\,6$ &  slope     \\
\cmidrule{1-5}
 $1/16$   & $2.80 \times 10^{-2}$ & --         & $1.22 \times 10^{-1}$ & --  \\
 $1/32$   & $1.14 \times 10^{-2}$ & $1.30$ & $4.41 \times 10^{-2}$	 & $1.47$ \\
 $1/64$   & $4.90 \times 10^{-3}$ & $1.21$ & $1.72 \times 10^{-2}$	 & $1.35$ \\
 $1/128$ & $2.30 \times 10^{-3}$ & $1.09$ &  $7.64 \times 10^{-3}$ & $1.17$ \\
 $1/256$ & $1.11 \times 10^{-3}$ & $1.05$ &  $3.59 \times 10^{-3}$ & $1.09$ \\
\midrule
\vspace{9.2cm }
\end{tabular}\qquad
\includegraphics[height=9cm]{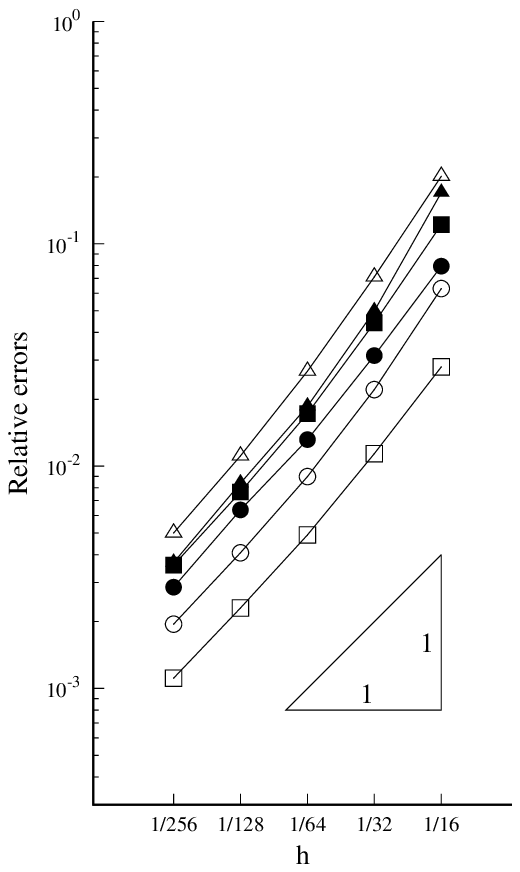}
\vspace{-9.2cm}
\captionsetup{labelformat=tableandpicture}
\caption{Errors and slopes for $(\nu,\varepsilon)=(10^{-1},10^{-1})$.}
\label{table:first_case}
\vspace*{5mm}
\end{figure}
\begin{figure}[!htbp]
\centering
\begin{tabular}{rrrrr}\toprule
  $h$       &   $Er\,1$                     &  slope  & $Er\,2$ & slope  \\  \midrule
  $1/16$  & $6.14 \times 10^{-2}$ & --     & $7.29 \times 10^{-2}$ & --         \\
 $1/32$   & $1.97 \times 10^{-2}$ & $1.64$ & $2.91 \times 10^{-2}$ & $1.33$  \\
 $1/64$   & $7.68 \times 10^{-3}$ & $1.36$ & $1.21 \times 10^{-2}$ & $1.26$  \\
 $1/128$ & $3.36 \times 10^{-3}$ & $1.19$ & $5.93 \times 10^{-3}$ & $1.03$  \\
 $1/256$ & $1.58 \times 10^{-3}$ & $1.09$ & $2.66 \times 10^{-3}$ & $1.15$  \\
\cmidrule{1-5}
 $h$  &  $Er\,3$ &  slope  & $Er\,4$ & slope \\
\cmidrule{1-5}
 $1/16$   & $2.50 \times 10^{-1}$ & --     & $2.06 \times 10^{-1}$ & -- \\
 $1/32$   & $9.14 \times 10^{-2}$ & $1.45$ & $6.08 \times 10^{-2}$ & $1.76$ \\
 $1/64$   & $3.31 \times 10^{-2}$ & $1.46$ & $2.11 \times 10^{-2}$ & $1.53$ \\
 $1/128$ & $1.28 \times 10^{-2}$ & $1.37$ & $8.78 \times 10^{-3}$ & $1.26$\\
 $1/256$ & $5.48 \times 10^{-3}$ & $1.23$ & $3.74 \times 10^{-3}$ & $1.23$\\
\cmidrule{1-5}
 $h$       &  $Er\,5$ &  slope  &  $Er\,6$ &  slope     \\
\cmidrule{1-5}
 $1/16$ & $5.01 \times 10^{-2}$ & --         & $5.38 \times 10^{-1}$ & --  \\
 $1/32$ & $1.92 \times 10^{-2}$ & $1.38$ & $2.54 \times 10^{-1}$ & $1.08$ \\
 $1/64$ & $7.53 \times 10^{-3}$ & $1.35$ & $1.05 \times 10^{-1}$ & $1.27$ \\
 $1/128$ & $3.28 \times 10^{-3}$ & $1.20$ &  $3.88 \times 10^{-2}$ & $1.44$ \\
 $1/256$ & $1.53 \times 10^{-3}$ & $1.10$ &  $1.35 \times 10^{-2}$ & $1.52$ \\
\midrule
\vspace{9.2cm }
\end{tabular}\qquad
\includegraphics[height=9cm]{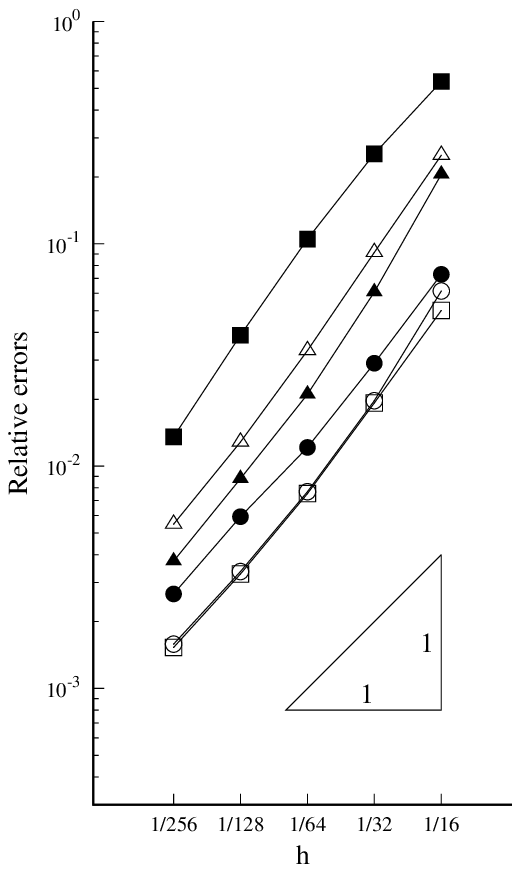}
\vspace{-9.2cm }
\captionsetup{labelformat=tableandpicture}
\caption{Errors and slopes for $(\nu,\varepsilon)=(10^{-1},10^{-3})$.}
\label{table:second_case}
\end{figure}
\begin{figure}[!htbp]
\centering
\begin{tabular}{rrrrr}
\toprule
 $h$  &  $Er\,1$ &  slope  & $Er\,2$ & slope  \\  \midrule
 $1/16$   & $4.51 \times 10^{-2}$ & --     & $5.83 \times 10^{-2}$ & --         \\
 $1/32$   & $1.42 \times 10^{-2}$ & $1.67$ & $2.36 \times 10^{-2}$ & $1.31$  \\
 $1/64$   & $4.53 \times 10^{-3}$ & $1.65$ & $9.85 \times 10^{-3}$ & $1.26$  \\
 $1/128$ & $1.52 \times 10^{-3}$ & $1.58$ & $4.89 \times 10^{-3}$ & $1.01$  \\
 $1/256$ & $5.72 \times 10^{-4}$ & $1.41$ & $2.10 \times 10^{-3}$ & $1.22$  \\
\cmidrule{1-5}
 $h$  &  $Er\,3$ &  slope  & $Er\,4$ & slope \\
\cmidrule{1-5}
 $1/16$   & $4.78 \times 10^{-1}$ & --     & $3.16 \times 10^{-1}$ & -- \\
 $1/32$   & $2.00 \times 10^{-1}$ & $1.26$ & $9.18 \times 10^{-2}$ & $1.79$ \\
 $1/64$   & $7.03 \times 10^{-2}$ & $1.51$ & $2.95 \times 10^{-2}$ & $1.64$ \\
 $1/128$ & $2.31 \times 10^{-2}$ & $1.60$ & $1.17 \times 10^{-2}$ & $1.33$\\
 $1/256$ & $8.04 \times 10^{-3}$ & $1.52$ & $5.01 \times 10^{-3}$ & $1.23$\\
\cmidrule{1-5}
 $h$       &  $Er\,5$ &  slope  &  $Er\,6$ &  slope     \\
\cmidrule{1-5}
 $1/16$ & $4.93 \times 10^{-2}$   & --         & $7.97 \times 10^{-1}$  & --  \\
 $1/32$ & $1.92 \times 10^{-2}$   & $1.36$ & $6.05 \times 10^{-1}$  & $0.40$ \\
 $1/64$ & $7.30 \times 10^{-3}$   & $1.39$ & $5.32 \times 10^{-1}$  & $0.19$ \\
 $1/128$ & $2.91 \times 10^{-3}$ & $1.33$ &  $4.04 \times 10^{-1}$ & $0.40$ \\
 $1/256$ & $1.24 \times 10^{-3}$ & $1.22$ &  $2.74 \times 10^{-1}$ & $0.56$ \\
\midrule
\vspace{9.2cm}
\end{tabular}\qquad
\includegraphics[height=9cm]{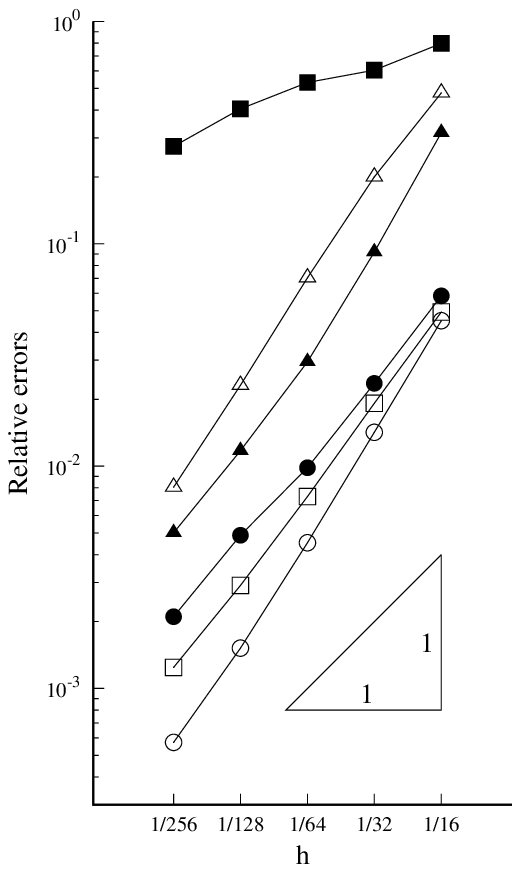}
\captionsetup{labelformat=tableandpicture}
\vspace{-9.2cm}
\caption{Errors and slopes for $(\nu,\varepsilon)=(1,0)$.}
\label{table:third_case}
\end{figure}
\begin{remark}
In the above the difference of $(\ucko_h, p_h, \Ccko_h)$ and $(\Pi_h\ucko, \Pi_hp, \Pi_h\Ccko)$ are computed.
For the difference of $(\ucko_h, p_h, \Ccko_h)$ and $(\ucko, p, \Ccko)$ see Appendix~\ref{subsec:appendix_numer}.
\end{remark}
%
\section{Conclusions}\label{sec:conclusions}
In this paper we have presented a linear stabilized {Lagrange}--{Galerkin} scheme~\eqref{lin_scheme} for the {Oseen}-type diffusive {Peterlin} viscoelastic model.
The scheme employs the conforming linear finite elements for all unknowns, velocity, pressure and conformation tensor, together with {Brezzi}--{Pitk\"{a}ranta}'s stabilization method.
In Theorems~\ref{thm:error_estimates} and~\ref{thm:error_estimates_pressure} we have established error estimates with the optimal convergence order under mild conditions, $\Delta t={\cal O}(1 / \sqrt{1+ |\log h|})$ for $d=2$ and $\Delta t ={\cal O}(\sqrt{h})$ for $d=3.$
They hold in the standard norms not only for the velocity and the conformation tensor but also for the pressure.
The theoretical convergence orders have been confirmed by two-dimensional numerical experiments.
\par
Although we have treated the stabilized scheme to reduce the number of degrees of freedom, the extension of the result to the combination of stable pairs for $(\ucko, p)$ and conventional elements for $\Ccko$ is straightforwards, e.g., P2/P1/P2 element.
In future we will extend this work to the Peterlin viscoelastic model with the nonlinear convective terms, and compare numerical results with other schemes in some benchmark problems.
\par
We recall that in our previous paper~\cite{LMNT-Peterlin_Oseen_Part_I}, Part~I, essentially unconditional stability and error estimates with the optimal convergence order were proved in two space dimensions.
There, our analysis allowed to include also the case $\varepsilon=0$.
%
%
%
%
%
\section*{Acknowledgements}
This research was supported by the German Science Agency (DFG) under the grants IRTG 1529 ``Mathematical Fluid Dynamics'' and TRR 146 ``Multiscale Simulation Methods for Soft Matter Systems'', and by the Japan Society for the Promotion of Science~(JSPS) under the Japanese-German Graduate Externship~``Mathematical Fluid Dynamics''.
H.M. was partially supported by the  German Academic Exchange Service (DAAD).
M.L.-M. and H.M. wish to thank B.~She (Czech Academy of Science, Prague) for fruitful discussion on the topic.
H.N. and M.T. are indebted to JSPS also for Grants-in-Aid for Young Scientists~(B), No.~26800091 and for Scientific Research~(C), No.~25400212 and Scientific Research~(S), No.~24224004, respectively.
{H.N.} is supported by Japan Science and Technology Agency~(JST), PRESTO.
%
%
%
%
%
\appendix
\renewcommand{\thesection}{A}
\setcounter{lemma}{0}
\renewcommand{\thelemma}{\thesection.\arabic{Lemma}}
\setcounter{remark}{0}
\renewcommand{\theremark}{\thesection.\arabic{Remark}}
\setcounter{figure}{0}
\renewcommand{\thefigure}{\thesection.\arabic{figure}}
\setcounter{equation}{0}
\makeatletter
  \renewcommand{\theequation}{%
  \thesection.\arabic{equation}}
  \@addtoreset{equation}{section}
\makeatother
%
%
%
%
\section*{Appendix}
\subsection{Proof of Lemma~\ref{lem:estimates_r_R}}\label{subsec:proof_r_R}
We prove only~\eqref{ieq:r3}, \eqref{ieq:r4}, \eqref{ieq:R4} and~\eqref{ieq:R8},
since \eqref{ieq:r1}, \eqref{ieq:r2} and~\eqref{ieq:R2} have been proved in Part~I~\cite{LMNT-Peterlin_Oseen_Part_I} and the other estimates are similarly obtained.
\par
\eqref{ieq:r3}, \eqref{ieq:r4} and~\eqref{ieq:R4} are obtained as follows:
\begin{align*}
\| \rcko_{h3}^n \|_{-1}
& \le \|(\tr \Ccko^n)(\Ccko^n-\Ccko^{n-1} + \Xicko^{n-1}-\Ecko_h^{n-1})\|_0
\le c_s \bigl( \|\Ccko^n-\Ccko^{n-1}\|_0 + \|\Xicko^{n-1}\|_0 + \|\Ecko_h^{n-1}\|_0 \bigr) \\
& \le c_s \bigl( \sqrt{\Delta t} \|\Ccko\|_{H^1(t^{n-1},t^n;L^2)} + \alpha_{32}h\|\Ccko^{n-1}\|_2 + \|\Ecko_h^{n-1}\|_0 \bigr) \\
& \le c_s^\prime \bigl( \|\Ecko_h^{n-1}\|_0 + \sqrt{\Delta t} \|\Ccko\|_{H^1(t^{n-1},t^n;L^2)} + h ),\\
\| \rcko_{h4}^n \|_{-1}
& \le \| [\tr (\Xicko^n-\Ecko_h^n)] \Ccko_h^{n-1} \|_0
\le c \|\Ccko_h^{n-1}\|_{0,\infty} \| \tr (\Xicko^n-\Ecko_h^n) \|_0\\
& \le c^\prime \|\Ccko_h^{n-1}\|_{0,\infty} (\| \Xicko^n \|_0 + \|\Ecko_h^n\|_0)
 \le c^\prime \|\Ccko_h^{n-1}\|_{0,\infty} ( \alpha_{32} h \| \Ccko^n \|_2 + \|\Ecko_h^n\|_0)\\
& \le c_s \|\Ccko_h^{n-1}\|_{0,\infty} ( \|\Ecko_h^n\|_0 + h),\\
\| \Rcko_{h4}^n \|_0
& = 2 \|(\nabla \ecko_h^n) \Ccko_h^{n-1}\|_0
 \le 2d \|\Ccko_h^{n-1}\|_{0,\infty} \|\nabla \ecko_h^n\|_0
 \le 2d \|\Ccko_h^{n-1}\|_{0,\infty} \|\ecko_h^n\|_1,
\end{align*}
where in the estimate of~$\| \Rcko_{h4}^n \|_0$ the inequality~$\|AB\|_0\le d \|A\|_{0,\infty}\|B\|_0$ for $A\in L^\infty(\Omega)^{d\times d}$ and~$B\in L^2(\Omega)^{d\times d}$ has been employed.
\par
Finally, \eqref{ieq:R8} is proved as
\begin{align*}
\| \Rcko_{h8}^n \|_0
& = \|[\tr (\Ccko_h^{n-1}+\hat{\Ccko}_h^{n-1})] (\tr \Ecko_h^{n-1}) \Ccko^n\|_0
\le c_s (\|\Ccko_h^{n-1}\|_{0,\infty} + \|\hat{\Ccko}_h^{n-1}\|_{0,\infty}) \| \Ecko_h^{n-1} \|_0
\\
& \le c_s^\prime (\|\Ccko_h^{n-1}\|_{0,\infty} + 1) \| \Ecko_h^{n-1} \|_0,
\end{align*}
where for the last inequality we have used the boundedness of~$\| \hat{\Ccko}_h^{n-1} \|_{0,\infty}$ obtained by the estimate
\begin{align*}
\| \hat{\Ccko}_h^{n-1} \|_{0,\infty}
&\le \| \hat{\Ccko}_h^{n-1} - \Pi_h \Ccko^{n-1} \|_{0,\infty} + \| \Pi_h \Ccko^{n-1} \|_{0,\infty}
\le \alpha_{21} D(h) \| \hat{\Ccko}_h^{n-1} - \Pi_h \Ccko^{n-1} \|_1 + \| \Ccko \|_{C(L^\infty)}
\\
&\le \alpha_{21} D(h) \bigl( \| \hat{\Ccko}_h^{n-1} - \Ccko^{n-1} \|_1 + \| \Ccko^{n-1} - \Pi_h \Ccko^{n-1} \|_1 \bigr) + \| \Ccko \|_{C(L^\infty)} \\
&\le \alpha_{21} D(h) \bigl( \alpha_{32}h \| \Ccko^{n-1} \|_2 + \alpha_{20}h\| \Ccko^{n-1} \|_2 \bigr) + \| \Ccko \|_{C(L^\infty)} \\
&\le \alpha_{21} hD(h) ( \alpha_{20} + \alpha_{32} ) \| \Ccko \|_{C(H^2)} + \| \Ccko \|_{C(L^\infty)} \\
\phantom{MMMMMMM}
&\le \alpha_{21} h_1 D(h_1) ( \alpha_{20} + \alpha_{32} ) \| \Ccko \|_{C(H^2)} + \| \Ccko \|_{C(L^\infty)} \le c_s.
\phantom{MMMMMMMMMMMMMM}
\qed
\end{align*}
\subsection{Difference of $(\ucko_h, p_h, \Ccko_h)$ and $(\ucko, p, \Ccko)$ in Example.}\label{subsec:appendix_numer}
In Section~\ref{sec:numerics} we have computed the difference of $(\ucko_h, p_h, \Ccko_h)$ and $(\Pi_h\ucko, \Pi_hp, \Pi_h\Ccko)$.
Here, we give additional information on the error between $(\ucko_h, p_h, \Ccko_h)$ and $(\ucko, p, \Ccko)$.
We introduce a numerical integration formula of degree five with seven quadrature points for each triangle, 
and we denote the norm derived by the formula by adding the prime to the corresponding norm, 
\begin{align*}
\| \psi \|_{L^2(\Omega)'}
& \defeq 
\biggl\{
\sum_{K \in \mathcal{T}_h}  \mbox{meas}(K)  \sum_{i=1}^7 |\psi (a_i^{K})|^2 \, w_i \,  
\biggr\}^{1/2}
\approx
\| \psi \|_{L^2(\Omega)},
\end{align*}
where $\{( a_i^K, w_i) \}$ is a set of pairs of quadrature point and weight on $K \in \mathcal{T}_h$.  
When $\psi$ is a function in P1 finite element space, it holds that $\| \psi \|_{L^{2}(\Omega){'}} = \| \psi \|_{L^2(\Omega)}$. 
We abbreviate $\| \psi \|_{L^2(\Omega)'} $ as $\| \psi \|_{ L^{2}{'} } $. 
In the following the symbol $\prime$ means that the numerical integration is used in place of the exact integration.  
We define the relative errors~$Er\,k^\prime$, $k=1,\ldots, 6$, by
\begin{align*}
Er\,1^{\prime} & \defeq 
\fz{\|\ucko_h - \ucko\|_{\ell^\infty(L^{2}{'})}}{\|\Pi_h\ucko\|_{\ell^\infty(L^2)}}, 
&
Er\,2^{\prime} & \defeq 
\fz{\|\ucko_h - \ucko\|_{\ell^2(H^{1}{'})}}{\|\Pi_h\ucko\|_{\ell^2(H^1)}}, 
&
Er\,3^{\prime} & \defeq 
\fz{\|p_h - p\|_{\ell^2(L^{2}{'})}}{\|\Pi_h p\|_{\ell^2(L^2)}}, \\
Er\,4^{\prime} & \defeq 
\fz{|p_h - p|_{\ell^2(|\cdot|_h^{'})}}{\|\Pi_h p\|_{\ell^2(L^2)}}, 
&
Er\,5^{\prime} & \defeq 
\fz{\|\Ccko_h - \Ccko\|_{\ell^\infty(L^{2}{'})}}{\|\Pi_h \Ccko\|_{\ell^\infty(L^2)}}, 
&
Er\,6^{\prime} & \defeq 
\fz{\|\Ccko_h - \Ccko\|_{\ell^2(H^{1}{'})}}{\|\Pi_h\Ccko\|_{\ell^2(H^1)}}.
\end{align*}
\par
We deal with the case $(\nu, \varepsilon) = (10^{-1}, 10^{-1})$.
Table~\ref{table:Er_prime} shows the comparison of the values of~$Er\,k^\prime$ with those of~$Er\,k$, which reflects that convergence orders of $Er\,k$ are almost same with those of~$Er\,k^\prime$, though the values of~$Er\,2^\prime$ are about three to four times larger than $Er\,2$.
Therefore, the computation of the difference of $(\ucko_h, p_h, \Ccko_h)$ and $(\Pi_h \ucko, \Pi_h p, \Pi_h \Ccko)$ is sufficient in order to observe the behavior of convergence of $(\ucko_h, p_h, \Ccko_h)$ to $(\ucko, p, \Ccko)$.
\begin{table}[!htbp]
\centering
\small
\caption{Comparison of $Er\,k$ with $Er\,k^\prime$, $k=1,\ldots,6$, for $(\nu,\varepsilon)=(10^{-1},10^{-1})$.}
\label{table:Er_prime}
\begin{tabular}{rrrrrcrrrr}
\toprule
 $h$  &  $Er\,1$ & slope & $Er\,1^\prime$ & slope & \quad & $Er\,2$ &  slope  & $Er\,2^\prime$ & slope \\  \midrule
 $1/16$   & $6.29 \times 10^{-2}$ & --         & $8.15 \times 10^{-2}$ & --         & \quad & $7.94 \times 10^{-2}$ & --         & $1.94 \times 10^{-1}$ & --         \\
 $1/32$   & $2.21 \times 10^{-2}$ & $1.51$ & $2.68 \times 10^{-2}$ & $1.60$ & \quad & $3.14 \times 10^{-2}$ & $1.34$ & $9.20 \times 10^{-2}$ & $1.08$ \\
 $1/64$   & $8.98 \times 10^{-3}$ & $1.30$ & $1.02 \times 10^{-2}$ & $1.39$ & \quad & $1.32 \times 10^{-2}$ & $1.25$ & $4.54 \times 10^{-2}$ & $1.02$ \\
 $1/128$ & $4.07 \times 10^{-3}$ & $1.14$ & $4.40 \times 10^{-3}$ & $1.22$ & \quad & $6.35 \times 10^{-3}$ & $1.05$ & $2.27 \times 10^{-2}$ & $1.00$ \\
 $1/256$ & $1.95 \times 10^{-3}$ & $1.07$ & $2.03 \times 10^{-3}$ & $1.12$ & \quad & $2.86 \times 10^{-3}$ & $1.15$ & $1.12 \times 10^{-2}$ & $1.02$ \\
\midrule
 $h$  &  $Er\,3$ &  slope  & $Er\,3^\prime$ & slope & \quad & $Er\,4$ &  slope  & $Er\,4^\prime$ & slope \\
\hline
 $1/16$   & $2.02 \times 10^{-1}$ & --         & $2.13 \times 10^{-1}$ & --        & \quad & $1.70 \times 10^{-1}$ & --         & $1.81 \times 10^{-1}$ & -- \\
 $1/32$   & $7.11 \times 10^{-2}$ & $1.50$ & $7.38 \times 10^{-2}$ & $1.53$ & \quad & $4.99 \times 10^{-2}$ & $1.77$ & $5.21 \times 10^{-2}$ & $1.80$ \\
 $1/64$   & $2.67 \times 10^{-2}$ & $1.41$ & $2.73 \times 10^{-2}$ & $1.43$ & \quad & $1.86 \times 10^{-2}$ & $1.42$ & $1.90 \times 10^{-2}$ & $1.46$ \\
 $1/128$ & $1.11 \times 10^{-2}$ & $1.27$ & $1.12 \times 10^{-2}$ & $1.28$ & \quad & $8.39 \times 10^{-3}$ & $1.15$ & $8.44 \times 10^{-3}$ & $1.17$ \\
 $1/256$ & $5.01 \times 10^{-3}$ & $1.15$ & $5.03 \times 10^{-3}$ & $1.16$ & \quad & $3.69 \times 10^{-3}$ & $1.19$ & $3.69 \times 10^{-3}$ & $1.19$ \\
\midrule
 $h$  &  $Er\,5$ &  slope  & $Er\,5^\prime$ &  slope & \quad & $Er\,6$ &  slope  &  $Er\,6^\prime$ &  slope     \\
\hline
 $1/16$   & $2.80 \times 10^{-2}$ & --         & $2.80 \times 10^{-2}$ & --        & \quad & $1.22 \times 10^{-1}$ & --         & $1.64 \times 10^{-1}$ & --  \\
 $1/32$   & $1.14 \times 10^{-2}$ & $1.30$ & $1.14 \times 10^{-2}$ & $1.30$ & \quad & $4.41 \times 10^{-2}$ & $1.47$ & $6.95 \times 10^{-2}$ & $1.24$ \\
 $1/64$   & $4.90 \times 10^{-3}$ & $1.21$ & $4.90 \times 10^{-3}$ & $1.21$ & \quad & $1.72 \times 10^{-2}$ & $1.35$ & $3.22 \times 10^{-2}$ & $1.11$ \\
 $1/128$ & $2.30 \times 10^{-3}$ & $1.09$ & $2.30 \times 10^{-3}$ & $1.09$ & \quad & $7.64 \times 10^{-3}$ & $1.17$ & $1.56 \times 10^{-2}$ & $1.04$ \\
 $1/256$ & $1.11 \times 10^{-3}$ & $1.05$ & $1.11 \times 10^{-3}$ & $1.05$ & \quad & $3.59 \times 10^{-3}$ & $1.09$ & $7.69 \times 10^{-3}$ & $1.02$ \\
\bottomrule
\end{tabular}
\end{table}
%
%
%
%
%
 \newcommand{\noop}[1]{}


\begin{thebibliography}{10}
\bibitem{AboMatWeb-2002}
M.~Aboubacar, H.~Matallah, and M.F. Webster.
\newblock Highly elastic solutions for {O}ldroyd-{B} and
  {P}han-{T}hien/{T}anner fluids with a finite volume/element method: planar
  contraction flows.
\newblock {\em Journal of Non-Newtonian Fluid Mechanics}, 103:65--103, 2002.

\bibitem{AOP-2003}
M.A.~Alves, P.J.~Oliveira, and F.T.~Pinho.
\newblock Benchmark solutions for the flow of Oldroyd-B and PTT fluids in planar contractions.
\newblock{\em Journal Non-Newtonian Fluid Mechanics}, 110:45--75, 2003.

\bibitem{BaSa-1992}
J.~Baranger and D.~Sandri.
{Finite element approximation of viscoelastic fluid flow.}
\newblock{\em Numerische Mathematik}, 63:13--27, 1992.

\bibitem{BaSu-2012}
J.W.~Barrett and E.~S\"{u}li.
{Existence and equilibration of global weak solutions to kinetic models for dilute polymers {II}: {H}ookean-type
  models}.
\newblock{\em Mathematical Models and Methods in Applied Sciences}, 22, 1150024, 2012.

\bibitem{BAB-1991}
A.V.~Bhave, R.C.~Armstrong, and R.A.~Brown.
\newblock Kinetic theory and rheology of dilute, nonhomogeneous polymer solutions.
\newblock{\em Journal~of~Chemical~Physics}, 95:2988--3000, 1991.

\bibitem{BonClePic-2007}
A.~Bonito, P.~Cl\'{e}ment, and M.~Picasso.
\newblock Mathematical and numerical analysis of a simplified time-dependent
  viscoelastic flow.
\newblock {\em Numerische Mathematik}, 107:213--255, 2007.

\bibitem{BonPicLas-2006}
A.~Bonito, M.~Picasso, and M.~Laso.
\newblock Numerical simulation of 3{D} viscoelastic flows with free surfaces.
\newblock {\em Journal of Computational Physics}, 215:691--716, 2006.

\bibitem{BPS-2001}
J.~Bonvin, M.~Picasso, and R.~Stenberg.
\newblock GLS and EVSS methods for a three-field Stokes porblem arising from viscoelastic flows.
\newblock{\em Computer Methods in Applied Mechanics and Engineering}, 190:3893--3941, 2001.

\bibitem{BouMadMetRaz-1997}
K.~Boukir, Y.~Maday, B.~M\'etivet, and E.~Razafindrakoto.
\newblock A high-order characteristics/finite element method for the incompressible {N}avier--{S}tokes equations.
\newblock {\em International Journal for Numerical Methods in Fluids}, 25:1421--1454, 1997.

\bibitem{BoyLelMan-2009}
S.~Boyaval, T.~Leli\`{e}vre, and C.~Mangoubi.
\newblock Free-energy-dissipative schemes for the {O}ldroyd-{B} model.
\newblock {\em ESAIM:~M2AN}, 43:523--561, 2009.

\bibitem{BreSco-2008}
S.C. Brenner and L.R. Scott.
\newblock {\em {T}he {M}athematical {T}heory of {F}inite {E}lement {M}ethods}.
\newblock Springer, New York, 3rd edition, 2008.

\bibitem{BreDou-1988}
F.~Brezzi and J.~{Douglas~Jr.}
\newblock Stabilized mixed methods for the {S}tokes problem.
\newblock {\em Numerische Mathematik}, 53:225--235, 1988.

\bibitem{BrePit-1984}
F.~Brezzi and J.~Pitk\"{a}ranta.
\newblock On the stabilization of finite element approximations of the {S}tokes
  equations.
\newblock In W.~Hackbusch, editor, {\em {E}fficient {S}olutions of {E}lliptic
  {S}ystems}, pages 11--19, Wiesbaden, 1984. Vieweg.

\bibitem{Cia-1978}
P.G. Ciarlet.
\newblock {\em {T}he {F}inite {E}lement {M}ethod for {E}lliptic {P}roblems}.
\newblock North-Holland, Amsterdam, 1978.

\bibitem{CroKeu-1982}
M.J. Crochet and R.~Keunings.
\newblock Finite element analysis of die swell of a highly elastic fluid.
\newblock {\em Journal of Non-Newtonian Fluid Mechanics}, 10:339--356, 1982.

\bibitem{deFruGar-2016}
J.~de~Frutos and B.~Garc{\' i}a-Archilla.
\newblock{Grad-div stabilization for the evolutionary {O}seen problem with inf-sup stable finite elements}.
\newblock{\em Journal of Scientific Computing}, 66:991--1024, 2016.

\bibitem{DL-2009}
P.~Degond and H.~Liu.
\newblock{Kinetic models for polymers with inertial effects}.
\newblock{\em Network and Heterogeneous Media}, 4:625--647, 2009.

\bibitem{Yu-1997}
Y.~Fan.
\newblock A comparative study of the discontinuous {G}alerkin and continuous SUPG finite element methods for computation of viscoelastic flows.
\newblock {\em Computational Methods in Applied Mechanics and Engineering}, 141:47--65, 1997.

\bibitem{FanTanPha-1999}
Y.~Fan, R.I.~Tanner and N.~Phan-Thien.
\newblock{{G}alerkin/least-square finite-element methods for steady viscoelastic flows}.
\newblock{\em  Journal of Non-Newtonian Fluid Mechanics}, 84:233--256, 1999.

\bibitem{FatKup-2004}
R.~Fattal and R.~Kupferman.
\newblock Constitutive laws for the matrix-logarithm of the conformation
  tensor.
\newblock {\em Journal of Non-Newtonian Fluid Mechanics}, 123:281--285, 2004.

\bibitem{FatKup-2005}
R.~Fattal and R.~Kupferman.
\newblock Time-dependent simulation of viscoelastic flows at high {W}eissenberg
  number using the log-conformation representation.
\newblock {\em Journal of Non-Newtonian Fluid Mechanics}, 126:23--37, 2005.

\bibitem{ForFor-1989}
M.~Fortin and A.~Fortin.
\newblock A new approach for the FEM simulations of viscoelastic flows.
\newblock {\em Journal of Non-Newtonian Fluid Mechanics}, 32:295--310, 1989.

\bibitem{FraSte-1991}
L.P. Franca and R.~Stenberg.
\newblock Error analysis of some {G}alerkin least squares methods for the
  elasticity equations.
\newblock {\em SIAM Journal on Numerical Analysis}, 28:1680--1697, 1991.

\bibitem{GuFo-1995}
R.~Gu\'enette and M.~Fortin.
\newblock A new mixed finite element method for computing viscoelastic flows.
\newblock {\em Journal of Non-Newtonian Fluid Mechanics}, 60:27--52, 1995.

\bibitem{HeyRan-1990}
J.G.~Heywood and R.~Rannacher.
\newblock Finite-element approximation of the nonstationary {N}avier--{S}tokes problem. {P}art {IV}: error analysis for second-order time discretization.
\newblock {\em SIAM Journal on Numerical Analysis}, 27(2):353--384, 1990.


\bibitem{Keu-1986}
R.~Keunings.
\newblock On the high {W}eissenberg number problem.
\newblock {\em Journal of Non-Newtonian Fluid Mechanics}, 20:209--226, 1986.

\bibitem{LeeXu-2006}
Y.-J. Lee and J.~Xu.
\newblock New formulations, positivity preserving discretizations and stability
  analysis for non-{N}ewtonian flow models.
\newblock {\em Computer Methods in Applied Mechanics and Engineering},
  195:1180--1206, 2006.

\bibitem{LeeXuZha-2011}
Y.-J. Lee, J.~Xu, and C.-S. Zhang.
\newblock Global existence, uniqueness and optimal solvers of discretized viscoelastic flow models.
\newblock {\em Mathematical Models and Methods in Applied Sciences}, 21(8):1713--1732, 2011.

\bibitem{LMNT-Peterlin_Oseen_Part_I}
M.~Luk\'{a}\v{c}ov\'{a}-Medvi\v{d}ov\'{a}, H.~Mizerov\'{a}, H.~Notsu, and M.~Tabata.
\newblock Numerical analysis of the {Oseen}-type {Peterlin} viscoelastic model by the stabilized {L}agrange--{G}alerkin method, {P}art~{I}: A nonlinear scheme.
\newblock {\em ESAIM:~M2AN}, in press.
\newblock DOI: 10.1051/m2an/2016078.

\bibitem{LukMizNec-2015}
M.~Luk\'{a}\v{c}ov\'{a}-Medvi\v{d}ov\'{a}, H.~Mizerov\'{a}, and \v{S}.
  Ne\v{c}asov\'{a}.
\newblock Global existence and uniqueness result for the diffusive {P}eterlin
  viscoelastic model.
\newblock {\em Nonlinear Analysis: Theory, Methods \& Applications},
  120:154--170, 2015.

\bibitem{LMNR-2016}
M.~Luk\'{a}\v{c}ov\'{a}-Medvi\v{d}ov\'{a}, H.~Mizerov\'a, \v{S}.~Ne\v{c}asov\'a, and M.~Renardy.
\newblock Global existence result for the generalized Peterlin
viscoelastic model.
\newblock Submitted to {\em SIAM Journal of Mathematical Analysis}, 2016.

\bibitem{LNS-2015}
M.~Luk\'{a}\v{c}ov\'{a}-Medvi\v{d}ov\'{a}, H.~Notsu, and B.~She.
\newblock Energy dissipative characteristic schemes for the diffusive
  {O}ldroyd-{B} viscoelastic fluid.
\newblock {\em International Journal for Numerical Methods in Fluids}, 81:523--557, 2016.
\newblock DOI:~10.1002/fld.4195.

\bibitem{Nec-1967}
J.~Ne\v{c}as.
\newblock {\em {L}es {M}\'ethods {D}irectes en {T}h\'eories des
  {\'E}quations {E}lliptiques}.
\newblock Masson, Paris, 1967.

\bibitem{MarCro-1987}
J.M. Marchal and M.J. Crochet.
\newblock A new mixed finite element for calculating viscoelastic flow.
\newblock {\em Journal of Non-Newtonian Fluid Mechanics}, 26:77--114, 1987.

\bibitem{Miz-2015}
H.~Mizerov\'{a}.
\newblock Analysis and numerical solution of the {P}eterlin viscoelastic model.
\newblock 2015.
\newblock PhD thesis, University of Mainz, Germany.

\bibitem{NadSeq-2007}
L.~Nadau and A.~Sequeira.
\newblock Numerical simulations of shear-dependent viscoelastic flows with a
  combined finite element-finite volume method.
\newblock {\em Computers \& Mathematics with Applications}, 53:547--568, 2007.

\bibitem{NT-NCP}
H.~Notsu and M.~Tabata.
\newblock Error estimates of stable and stabilized {L}agrange--{G}alerkin
  schemes for natural convection problems.
\newblock {\em {\rm {a}rXiv:1511.01234~[math.NA]}}.

\bibitem{NT-2015-JSC}
H.~Notsu and M.~Tabata.
\newblock Error estimates of a pressure-stabilized characteristics finite
  element scheme for the {O}seen equations.
\newblock {\em Journal of Scientific Computing}, 65(3):940--955, 2015.

\bibitem{NT-2016-M2AN}
H.~Notsu and M.~Tabata.
\newblock Error estimates of a stabilized {L}agrange--{G}alerkin scheme for the
  {N}avier--{S}tokes equations.
\newblock {\em ESAIM:~M2AN}, 50(2):361--380, 2016.

\bibitem{OlsReu-2003}
M.A.~Olshanskii and A.~Reusken.
\newblock Grad-div stabilization for {S}tokes equations.
\newblock {\em Mathematics of Computation}, 73(248):1699--1718, 2003.

\bibitem{OP-2002}
R.G.~Owens and T.N.~Philips.
\newblock Computational Rheology.
\newblock {\em Imperial College Press}, 2002.

\bibitem{OCP-2002}
R.G.~Owens, C.~Chauvi\`ere, and T.N.~Philips.
\newblock A locally-upwinded spectral technique (LUST) for viscoelastic flows.
\newblock {\em Journal of Non-Newtonian Fluid Mechanics}, 108:49--71, 2002.

\bibitem{Pet-1966}
A.~Peterlin.
\newblock Hydrodynamics of macromolecules in a velocity field with longitudinal gradient.
\newblock {\em Journal of Polymer Science Part~B: Polymer Letters}, 4:287--291, 1966.

\bibitem{PicRap-2001}
M.~Picasso and J.~Rappaz.
\newblock Existence, a priori and a posteriori error estimates for a nonlinear
  three-field problem arising from {O}ldroyd-{B} viscoelastic flows.
\newblock {\em ESAIM:~M2AN}, 35:879--897, 2001.

\bibitem{RajaArmBr-1990}
D.~Rajagopalan, R.C.~Armstrong and R.A.~Brown.
\newblock Finite element methods for calculation of steady, viscoelastic flow using constitutive equations with a {N}ewtonian viscosity.
\newblock {\em Journal of Non-Newtonian Fluid Mechanics}, 36:159--192, 1990.

\bibitem{Ren-2000}
M.~Renardy.
\newblock {\em {M}athematical {A}nalysis of {V}iscoelastic {F}lows}.
\newblock CBMS-NSF Conference Series in Applied Mathematics 73. SIAM, New York, 2000.

\bibitem{Ren-2008}
M.~Renardy.
\newblock Mathematical analysis of viscoelastic fluids.
\newblock In {\em {H}andbook of {D}ifferential {E}quations:~{E}volutionary
  {E}quations}, volume~4, pages 229--265, Amsterdam, 2008. North-Holland.

\bibitem{Ren-2010}
M.~Renardy.
\newblock The mathematics of myth: Yield stress behaviour as a limit of
  non-monotone constitutive theories.
\newblock {\em Journal of Non-Newtonian Fluid Mechanics}, 165:519--526, 2010.

\bibitem{RenWan-2015}
M.~Renardy and T.~Wang.
\newblock Large amplitude oscillatory shear flows for a model of a thixotropic
  yield stress fluid.
\newblock {\em Journal of Non-Newtonian Fluid Mechanics}, 222:1--17, 2015.

\bibitem{RuiTab-2002}
H.~Rui and M.~Tabata.
\newblock A second order characteristic finite element scheme for
  convection-diffusion problems.
\newblock {\em Numerische Mathematik}, 92:161--177, 2002.

\bibitem{Suli-1988}
E.~S\"uli.
\newblock Convergence and nonlinear stability of the {L}agrange--{G}alerkin method for the {N}avier--{S}tokes equations.
\newblock {\em Numerische Mathematik}, 53:459--483, 1988.

\bibitem{TabUch-2015-NS}
M.~Tabata and S.~Uchiumi.
\newblock An exactly computable {L}agrange--{G}alerkin scheme for the {N}avier--{S}tokes equations and its error estimates.
\newblock {\em Mathematics of Computation}, in press.
\newblock DOI: 10.1090/mcom/3222.

\bibitem{W-2013}
Wang~K.
\newblock A new discrete {EVSS} method for the viscoelastic flows.
\newblock{\em Computers and Mathematics with Applications}, 65:609--615, 2013.

\bibitem{WapKeuLeg-2000}
P.~Wapperom, R.~Keunings, and V.~Legat.
\newblock The backward-tracking {L}agrangian particle method for transient
  viscoelastic flows.
\newblock {\em Journal of Non-Newtonian Fluid Mechanics}, 91:273--295, 2000.

\end{thebibliography}
\end{document}